\documentclass[11pt]{amsart}

\usepackage{amssymb,amsfonts}
\usepackage[all,arc]{xy}
\usepackage{enumerate}
\usepackage{mathrsfs}
\usepackage{graphicx}
\usepackage{caption}
\usepackage{subcaption}
\usepackage{subfig}
\usepackage{ textcomp }
\usepackage{ upgreek }
\usepackage{tikz}
\usepackage{leftidx}
\usepackage{enumitem}
\usepackage{nameref}
\usepackage[english]{babel}
\usetikzlibrary{calc}
\newtheorem{thm}{Theorem}[section]
\newtheorem{cor}[thm]{Corollary}
\newtheorem{prop}[thm]{Proposition}
\newtheorem{lem}[thm]{Lemma}

\theoremstyle{definition}
\newtheorem{defn}[thm]{Definition}

\newtheorem{theo}{Theorem}
\newcounter{tmp}
\def\<{\langle}
\def\>{\rangle}

\theoremstyle{remark}
\newtheorem{rem}[thm]{Remark}

\makeatletter
\let\c@equation\c@thm
\makeatother
\numberwithin{equation}{section}

\date{\today}

\begin{document}
\author[Chinyere]{I. Chinyere}
\address{ Ihechukwu Chinyere\\
Department of Mathematics and Maxwell Institute for Mathematical Sciences\\
Heriot--Watt University\\
Edinburgh EH14 4AS }
\email{ic66@hw.ac.uk}
\author[Howie]{J. Howie}
\address{ James Howie\\
Department of Mathematics and Maxwell Institute for Mathematical Sciences\\
Heriot--Watt University\\
Edinburgh EH14 4AS }
\email{J.Howie@hw.ac.uk}

   \title[Induced one-relator products]{On One-relator  products induced by generalised triangle groups}

   \begin{abstract}
    In this paper we study a group $G$ which is the quotient of a free product of groups by the normal
closure of a word that is contained in a in a subgroup which has the form of a generalised triangle group.
 We use known properties of generalised triangle groups, together with detailed analyses of pictures and of words in free monoids, 
to prove a number of results such as a Freiheitssatz and the existence of Mayer-Vietoris sequences for such groups under suitable hypotheses.  The hypotheses are weaker than those in an earlier article of Howie and Shwartz, yielding generalisations in two directions of the results in that article.
   \end{abstract}
   \date{\today}
   \maketitle

\section{Introduction}
A one-relator product $G$ of groups $G_{_1}$ and $G_{_2}$ is the quotient of the free product $G_{_1} * G_{_2}$ by the normal closure of a single element $W$ which has free product length at least two. 
  Such groups are natural generalisations of one-relator groups, and have been the subject of several articles generalising results from the rich theory of one-relator groups.
In this paper we consider the case where $W=R^n$ is a proper power and contained in a subgroup of $G_{_1} * G_{_2}$ the form $A*B$ with $A$ and $B$ conjugate to cyclic subgroups of $G_{_1}$ and/or $G_{_2}$. We can of course assume that $A$ is generated by some element $a$ and $B$ is generated by $UbU^{-1}$ for some word $U$ and letter  $b$ in $G_{_1} * G_{_2}$. Hence $R$ is a word in $\lbrace a, UbU^{-1}\rbrace$. As in \cite{hs}, we also require in some cases the technical condition that $(a,b)$ be \textit{admissible}: whenever both $a$ and $b$ belong to same factor, say $G_{_1}$, then either the subgroup of $G_{_1}$ generated by $\lbrace a, b\rbrace$ is cyclic or $\langle a\rangle \cap\langle b\rangle=1$.

\medskip
A generalised triangle group is a one-relator product of two finite cyclic groups in which the relator is a proper power. For the sake of this paper however, we will allow the cyclic groups to be infinite. If a one-relator product is in the form of $G$ described above, then we say  $G$ is induced by a generalised triangle group. In such a case $G$ can be realised as a push-out of groups as shown in Figure \ref{push-out}  below where $H$ is the corresponding generalised triangle group.

\medskip
The concept of one-relator product induced by a generalised
triangle group was introduced in \cite{hs}, where a number
of results were proved under the hypotheses that $n\ge 3$ and that the pair $(a,b)$ is admissible.  In the present paper we
prove similar results under hypotheses that are in general weaker than those assumed in \cite{hs}.

\medskip\noindent{\bf Hypothesis A}. $n\ge 2$, $R$
has free-product length at least $4$ as a word in the free product $\< a\>*\< UbU^{-1}\>$, and the pair $(a,b)$ is admissible.

\medskip\noindent{\bf Hypothesis B}. $n\ge 2$
and no letter of $R$ has order $2$ in $G_{_1}$ or $G_{_2}$.

\medskip
Under either of the above hypotheses, we prove the following:

\begin{thm}\label{corro 1}
Let $H$ be the generalised triangle group inducing the one relator product $G$. Then the maps  $G_{_1}\rightarrow G$, $G_{_2}\rightarrow G$ and $H\rightarrow G$ are all injective.
\end{thm}
\begin{thm}\label{corro 2}
If the word problems are soluble for $H$, $G_{_1}$ and $G_{_2}$, then it is soluble for $G$.
\end{thm}
\begin{thm}\label{corro 3}
If a cyclic permutation of $R^n$ has the form $W_{_1}W_{_2}$ with $0<\ell(W_{_1}),\ell(W_{_2})<\ell(R^n)$ as words in $G$, then $W_{_1}\neq 1\neq W_{_2}$ as words in $G$. In particular $R$ has order  $n$ in $G$.
\end{thm}

\begin{thm}\label{corro 4}
The pushout of groups in Figure 1 below is {\em geometrically
Mayer-Vietoris} in the sense of \cite{eh}.  In particular it
gives rise to Mayer-Vietoris sequences
$$\cdots \to H_{k+1}(G,M)\to H_k(A*B,M)\to$$ $$ H_k(G_{_1}*G_{_2},M)\oplus H_k(H,M)\to H_k(G,M)\to\cdots  $$
and $$\cdots \to H^k(G,M)\to H^k(G_{_1}*G_{_2},M)\oplus H^k(H,M)$$
$$\to H^k(A*B,M)\to H^{k+1}(G,M)\to\cdots $$ for any $\mathbb{Z}G$-module $M$.
\end{thm}

The first part of Theorem \ref{corro 1} is a generalisation of Magnus' Freiheitssatz for one-relator groups \cite{Mag1}.
There are many generalisations of the Freiheitssatz to one-relator products of a special nature.  Most relevant to the present paper, it was proved for arbitrary one-relator products in \cite{jh1,jh2}, provided $n\ge 4$. Theorem \ref{corro 2} was proved by Magnus \cite{Mag2} for one-relator groups. Known versions for one-relator products include the case when $n\ge 4$ \cite{jh3}. Theorem \ref{corro 3} is a version of a result of Weinbaum
\cite{Wein} for one-relator groups. All four of these results were proved in \cite{hs} under the hypotheses that $n\ge 3$ and that the pair $(a,b)$ is admissible.

\medskip
The above theorems will  proven using  the notion of {\em clique-picture}
from \cite{hs} and Theorems \ref{thmm1} and \ref{thmm2} below.
\begingroup
\setcounter{tmp}{\value{theo}}
\setcounter{theo}{0} 
\renewcommand\thetheo{\Alph{theo}}
\begin{theo}\label{thmm1}
 If Hypothesis A above holds, then a minimal  clique-picture over $G$ satisfies  the small-cancellation condition $C(6)$.
\end{theo}

\begin{theo}\label{thmm2}
If Hypothesis B above holds, then a minimal  clique-picture over $G$ satisfies  the small-cancellation condition $C(6)$.
\end{theo}
\endgroup

\medskip
 These results naturally lead one to speculate that similar results apply for any relator $R^n$ with $n\ge 2$, without any of the restrictions in Hypotheses A and B.
A less ambitious speculation might be that the condition of
admissibility could be dropped from Hypothesis A.
We are not aware of any counterexamples to such speculation, but the methods currently available to us do not permit us to
weaken either of the hypotheses.

\medskip
The rest of the paper is arranged in the following way. In Section \ref{sec2} we define some of the terminologies that are used in the paper. In Section \ref{sec3} we consider the various possible generalised triangle group descriptions and how they are related via push-out diagrams. In Section \ref{sec4} we recall the idea of pictures and clique-pictures. In Section \ref{sec5} we set-up some notations as well as prove general preliminary results about clique-pictures. Sections \ref{sec6} and \ref{sec7} contains various Lemmas that are used to prove Theorems \ref{thmm1} and \ref{thmm2}. Section \ref{sec8} is the final section. There we give a proof of Theorems \ref{corro 1}- \ref{corro 4}.

\medskip
It is worth mentioning at this point that most of the arguments in this paper are adaptations of the ones found in \cite{jh2}, \cite{jh1} and \cite{hs}. The reader can consult those for better understanding.
\section{Preliminaries}\label{sec2}
In this section we give some basic definitions and results on periodic words in a free monoid.

\medskip
Let $I$ be an indexing set and $S=\lbrace z_{_i}\rbrace_{i\in I}$ be a set and $S^*$ the free monoid on $S$. Each $z_{_i}$ is called a \textit{letter}. We assume $S$ is equipped with an involution $z_{_i} \mapsto {z_{_i}}^{-1}$. A word in $S^*$ is just a collection of letters. Let $w=z_{_1} z_{_2} \cdots z_{_n}\in S^*$ for some integer $n\geq0$. Then $n$ is called the length of $w$ and is denoted by $\ell(w)$.  A \textit{segment} of $w$ is a collection of consecutive letters in $w$. A segment is called \textit{initial} if it has the form $z_{_1} z_{_2} \cdots z_{_k}$ for some $k\leq n$ and terminal if it is of the form $z_j z_{_{j+1}} \cdots z_{_n}$ for $j\geq1$. We call a segment of $w$ \textit{proper} if it misses at least one letter in $w$. Let $u=z_{_i} z_{i+1} \cdots z_{_{i+t}}$ and $v=z_j z_{_{j+1}} \cdots z_{_{j+s}}$ be segments of $w$. Suppose without loss of generality that $i+t\leq j+s$, we say $u$ and $v$ \textit{intersect} if $j\leq i+t$. In which case the \textit{intersection} is the segment $u$ if $j\leq i$ and $z_j z_{_{j+1}} \cdots z_{_{i+t}}$ otherwise. The \textit{involute} of $w$ is ${w}^{-1}={z_{_n}}^{-1} {z_{_{n-1}}}^{-1} \cdots {z_{_1}}^{-1}$. We call $w$ a \textit{proper power} if it has the form $w=u^t$ for some proper initial segment $u$ of $w$. A \textit{cyclic permutation} of $w$ is a word of the form $z_{_{\rho(1)}} z_{_{\rho(2)}} \cdots z_{_{\rho(n)}}$ where $\rho$ is some power of the permutation $(1 ~2 \cdots ~ n)$. A \textit{proper cyclic permutation} is one in which $\rho$ is not the identity. Two words $u$ and $v$ are said to be \textit{identically equal}, written $u\equiv v$,  if they are equal in $S^*$. The notation $=$ will be reserved for equality in some quotient group.

\begin{defn}
A word $w$ of length $n$ has a period $\gamma$ if $\gamma\leq n$  and $w_{_i}=w_{_{i+\gamma}}$ for all $i\leq n-\gamma$.
\end{defn}

\begin{defn}\label{pe}
A word $w$ is said to be bordered by $u$ and $v$  if   $u$ and $v$ are proper initial  and a terminal segments of $w$ respectively. Furthermore we say that $w$ is \textit{bordered} by $u$ if $u\equiv v$.
\end{defn}

\begin{rem}\label{bor}
It follows immediately that a $w$ bordered by $u$ has period $\gamma=\ell(w)-\ell(u)$.
\end{rem}

\begin{thm}\cite{FW}\label{tm}
Let $w$ be a word having periods $\gamma$ and $\rho$ with $\rho\leq\gamma$. If $\ell(w)\geq \gamma + \rho-gcd(\gamma,\rho)$, then $w$ has period $gcd(\gamma,\rho)$.
\end{thm}

\begin{cor}\label{ctm}
Let $w$ be a word having initial segment $w_{_1}$ with period $\gamma$ and terminal segment $w_{_2}$ with period $\rho$. If $w_{_1}$ and $w_{_2}$ intersect in a segment $u$ with $\ell(u)\geq \gamma + \rho-gcd(\gamma,\rho)$, then $w$ has period $gcd(\gamma,\rho)$.
\end{cor}

\begin{lem}\label{nz2}
Suppose that $W\in S^*$ is a cyclically reduced word of the form
$x_1V_1y_1V_1^{-1}=z_0z_1\cdots z_{2k-1}$ for some letters $x_1,y_1$ and some word $V_1$, where $\ell(W)=2k$. Suppose also that $W$ has a cyclic permutation
of the form $z_jz_{j+1}\cdots z_{2k-1}z_0\cdots z_{j-1}=x_2V_2y_2V_2^{-1}$, for some letters $x_2,y_2$ and some word $V_2$, where $j \not\equiv 0$ mod $k$.  Then one of the following holds:
\begin{enumerate}
\item $\{x_1,y_1\}=\{x_2,y_2\}$ and 
$$ W\equiv  \prod_{j=1}^s [x_1^{\alpha(j)}V_3y_1^{\beta(j)}V_3^{-1}]$$ for some odd integer $s>1$ and some word $V_3$, with $\alpha(j),\beta(j)=\pm 1$ for each $j$.
\item $y_i=x_i^{-1}$ for $i=1,2$, and 
$$ W\equiv \prod_{j=1}^s [x_1^{\alpha(j)}V_3x_2^{\beta(j)}V_3^{-1}]$$ for some even integer $s>0$ and some word $V_3$, with $\alpha(j),\beta(j)=\pm 1$ for each $j$.
\end{enumerate}
\end{lem}

\begin{proof}
Write $x_1V_1y_1V_1^{-1}=z_0z_1\cdots z_{2k-1}$ where $\ell(W)=2k$.  Then $x_2V_2y_2V_2^{-1}$ has the form
$z_jz_{j+1}\cdots z_{2k-1}z_0\cdots z_{j-1}$ for some $j\in\{1,\dots,2k-1\}$.  Thus $z_i=z_{2k-i}^{-1}$ unless $i\equiv 0$ mod $k$, and $z_i=z_{2j-i}^{-1}$ unless $i\equiv j$ mod $k$.  Let $m=\gcd(j,k)$ and let $V_3=z_1z_2\cdots z_{m-1}$.  We interpret all subscripts modulo $2k$.  For
$i\not\equiv 0$ mod $m$, we have $z_i=z_{2j-i}^{-1}=z_{i-2j}$
and also $z_i=z_{i+2k}$, and so $z_i=z_{i+2m}$.  It follows that $x_1V_1y_1V_1^{-1}=\prod_{t=1}^s [\xi_tV_3\eta_tV_3^{-1}]$ for some letters $\xi_t,\eta_t$, where $s=k/m$.  By hypothesis $j\not\equiv 0$ mod $k$, and so $s>1$.

Suppose first that $s$ is odd.  Replacing $j$ by $k+j$ if necessary, we may assume that $j/m$ is also odd.  We have a chain of equalities $z_0=z_{2j}^{-1}=z_{2k-2j}=\cdots$
that continues until it reaches $z_d^{\pm 1}$ for some subscript $d\in\{j,k,j+k\}$.  Since the equalities link 
letters with subscripts of the same parity, we must have $d=j+k$. Moreover, every $z_e$ with $e\equiv 0$ mod $2m$ appears in the chain.  There are precisely $s$ such letters, so an even number of equalities, and hence $\xi_1=x_1=z_0=z_{j+k}=y_2$ and $\xi_t=x_1^{\pm 1}=y_2^{\pm 1}$ for each $t$.  By a similar argument $\eta_t=y_1^{\pm1}=x_2^{\pm1}$ for each $t$.

Now suppose that $s$ is even.  Then $j/m$ is odd.  Arguing as above, we have a chain of $s-1$ equalities $z_0=z_{2j}^{-1}=\cdots$, which must end with $z_k^{-1}$, and a similar chain of equalities equating $z_j$ with $z_{j+k}^{-1}$.  Hence in this case $\xi_t=x_1^{\pm1}=y_1^{\mp1}$ for each $t$, and $\eta_t=x_2^{\pm1}=y_2^{\mp1}$ for each $t$.
\end{proof}

\section{Refinements}\label{sec3}

Let $R$ be a cyclically reduced word of length at least $2$ in the free product $G_{_1} * G_{_2}$. As mentioned in the introduction, we are interested in the case where $R$ is contained in the subgroup $A*B$ where $A$ and $B$ are cyclic subgroups of conjugates of $G_{_1}$ or $G_{_2}$. Let $N(R^n)$ denote the normal closure of $R^n$ in $G_{_1} * G_{_2}$ with $n\geq2$. Then the group of interest is the following:
\[G=\dfrac{(G_{_1} * G_{_2})}{N(R^n)}\]
If $S$ and $T$ are the generators of $A$ and $B$ respectively, then we can construct a generalised triangle group $H=\<~x,~y~|~x^p,~y^q,~R'(x,y)^n~\>$ where $R'(S,T)$ is identically equal to $R$ in $G_{_1} * G_{_2}$. We can then realise $G$ as the push-out in Figure \ref{push-out}. We will call the set $\lbrace S, T\rbrace$ or Figure \ref{push-out} the choice of \textit{generalised triangle group description} for $G$
\begin{figure}[h!]
\begin{tikzpicture}[>=latex]
\centering
\node (w) at (0,0) {\(A*B\)};
\node (x) at (0,-2) {\(G_{_1} * G_{_2}\)};
\node (y) at (2,0) {\(H\)};
\node (z) at (2,-2) {\(G\)};
\draw[->] (w) -- (y);
\draw[->] (w) -- (x);
\draw[->] (x) -- (z);
\draw[->] (y) -- (z);
\end{tikzpicture}
\caption{\textit{Push-out diagram.}}\label{push-out}
\end{figure}

\medskip
The subgroup $A*B$ may not be unique amongst all two-generator subgroups containing $R$. If $A'*B'$ is another two-generator subgroup of $G_{_1} * G_{_2}$ with suitable generating set $\lbrace S', T'\rbrace$ and $A*B\subseteq A'*B'$, then we can write $S$, $T$ and $R'(S,T)$ as words in this new generating set. In general we have that $\ell(S')+\ell(T')\leq \ell(S)+\ell(T)$ and  $\ell(R')\leq\ell(R'')$ where $R\equiv R'(S,T)\equiv R''(S',T')$. (Note that the lengths here are in terms of the new generating set). If any of the two inequalities is strict, we say that the generalised triangle group description given by $R''$ is a \textit{refinement} of the one given by $R'$.

\medskip
Let  $p'$ and $q'$ be the  orders of $S'$ and $T'$ respectively. Then we have that the group $H'=\<~x',~y'~|~x^{p'},~y^{q'},~R''(x,y)^n~\>$ and the refinement gives a commutative diagram as in Figure \ref{Double push-out}, in which both squares are push-outs.
\begin{figure}[h!]
\begin{tikzpicture}[>=latex]
\centering
\node (a) at (0,0) {\(C_{_{p'}}*C_{_{q'}}\)};
\node (b) at (2,0) {\(C_{_p}*C_{_q}\)};
\node (c) at (4,0) {\(G_{_1} * G_{_2}\)};
\node (d) at (0,-2) {\(H'\)};
\node (e) at (2,-2) {\(H\)};
\node (f) at (4,-2) {\(G\)};
\draw[->] (a) -- (b);
\draw[->] (b) -- (c);
\draw[->] (a) -- (d);
\draw[->] (d) -- (e);
\draw[->] (e) -- (f);
\draw[->] (c) -- (f);
\draw[->] (b) -- (e);
\end{tikzpicture}
\caption{\textit{Double push-out diagram.}}
\label{Double push-out}
\end{figure}
A generalised triangle group description for $G$ is said to be \textit{maximal} if no refinement is possible. In other words, $A*B$ is maximal amongst all two-generator subgroups of $G_{_1} * G_{_2}$ containing $R$. It follows from the inequality condition that maximal refinements always exist (but not necessarily unique). From now on we will assume that we are working with maximal refinement.

We next note a useful consequence in this context of Lemma \ref{nz2}.

\begin{lem}\label{refine}
Suppose that, in the above, the generators $S,T$ of $A,B$
have the forms $S=a$, $T=UbU^{-1}$ for some letters $a,b$ and some word $U$. Suppose that $(a,b)$ is an admissible pair.

If there are integers $\alpha,\beta,\gamma,\delta$ such that $a^\alpha Ub^\beta U^{-1}$ and $a^\gamma Ub^\delta U^{-1}$ are proper cyclic conjugates in $G_{_1}*G_{_2}$,
then a refinement is possible.
\end{lem}

\begin{proof}
We may apply Lemma \ref{nz2} to $a^\alpha Ub^\beta U^{-1}$ and $a^\gamma Ub^\delta U^{-1}$ except in the situation where $a^\alpha Ub^\beta U^{-1}\equiv b^\delta U^{-1}a^\gamma U$. Let us first consider this exceptional situation.  Then in particular $U\equiv U^{-1}$ and so $U$ has the form $VxV^{-1}$ for some word $V$ and some letter $x$ of order $2$. But we also have $a^\alpha=b^\delta$, and so by definition of admissibility $a,b$ have a common root $c$ say.  But then $A*B$ is a proper subgroup of $C*D$, where $C$ and $D$ are the cyclic subgroups of $G_{_1}*G_{_2}$ generated by $c,V$ respectively.  This is a refinement, as required.

Now apply Lemma \ref{nz2} with $x_1=a^\alpha$, $y_1=b^\beta$, $x_2=a^\gamma$, $y_2=b^\delta$, and $V_1=V_2=U$.  Consider
first the case when $s$ is odd in the conclusion of Lemma \ref{nz2}.
In this case, $A*B$ is a proper subgroup of $A*B'$, where $B'$ is the cyclic subgroup generated by $V_3bV_3^{-1}$, so
again we have a refinement.

Finally, suppose that $s$ is even in Lemma \ref{nz2}.  Then
$a^\alpha=b^{-\beta}$, so by admissibility $a,b$ have a common root $c$.  Then
$A*B$ is a proper subgroup of $C*D$, where $C,D$ are the cyclic subgroups generated by $c,V_3cV_3^{-1}$ respectively, so
again we have a refinement.

\end{proof}

\medskip
Before leaving this section, we mention a few important results about generalised triangle groups.

\begin{lem}\label{order}
For an integer $m>1$, a nontrivial element $X\in PSL_{_2}(\mathbb{C})$ has order $m$ if and only if $Tr(X)=2\cos\frac{\delta \pi}{m}$ for some $\delta$ satisfying $\text{gcd}(\delta,m)=1$.
\end{lem}
\begin{prop}\label{prop1}
Let $H=\<~x,~y~|~x^p,~y^q,~(x y)^2~\>$ be a triangle group.
If $v(x,y)=x^\alpha y^\beta x^\gamma y^\delta$ is trivial in $H$, with 
$\alpha,\gamma\in \lbrace 1, 2, \ldots , p-1\rbrace$ and $\beta,\delta\in \lbrace 1, 2, \ldots , q-1\rbrace$,
then one of the following holds:
\begin{enumerate}
\item $2\in\{p,q\}$;
\item $\alpha =\beta =\gamma =\delta=1$;
\item $\alpha=\gamma=p-1$ and $\beta=\delta=q-1$.
\end{enumerate}
\end{prop}
\begin{proof}
 Assume that $p\neq 2\neq q$ and consider the elements\\ $ X=\left( \begin{array}{cc}
e^{\frac{\pi i}{p}} & 0 \\
1 & e^{\frac{-\pi i}{p}} \end{array} \right)$  and $ Y=\left( \begin{array}{cc}
e^{\frac{\pi i}{q}} & t \\
0 & e^{\frac{-\pi i}{q}} \end{array} \right)$ in $PSL_{_2}(\mathbb{C})$.  By Lemma \ref{order}, $X$ and $Y$ have orders $p$ and $q$ respectively in $PSL_{_2}(\mathbb{C})$. If we take $t=-2\cos\left(\frac{\pi}{p}+\frac{\pi}{q}\right)$, then $Tr(X Y)=0$ and hence the map $x\mapsto X$, $y\mapsto Y$ extends to a faithful representation of $H$ in $PSL_{_2}(\mathbb{C})$. Suppose that $X^\alpha Y^\beta =Y^{-\delta} X^{-\gamma}$. By comparing the left lower entries of both sides of the equation we have 
 $\alpha\equiv\pm\gamma$ mod $p$ and
\[\sin \frac{\alpha \pi}{p}e^{\frac{\beta\pi i}{q}}\mp \sin \frac{\gamma \pi}{p}e^{\frac{\delta \pi i}{q}}=0.\] By expanding and solving component-wise, we have that \[\sin \frac{\gamma \pi}{p} \sin \left(\frac{\beta-\delta}{q} \right)\pi =0 .\]  In particular, $\beta=\delta$. Similarly we have 
 $\alpha=\gamma$. 

 Hence $v=(X^\alpha Y^\beta)^2=\pm I$.  By comparing off-diagonal entries, we see that $X^\alpha\neq\pm Y^{-\beta}$,
 so $v\neq +I$.  Hence $v=-I$, and so
 $Tr(X^\alpha Y^\beta)=0$, i.e 
\[2\cos\left(\frac{\alpha}{p}+\frac{\beta}{q}\right)\pi + t\dfrac{\sin\frac{\alpha\pi}{p} \sin\frac{\beta\pi}{q}}{\sin\frac{\pi}{p}\sin\frac{\pi}{q}}=0 .\]
Hence we obtain 
\[\tan\frac{\alpha \pi}{p} \tan\frac{\beta \pi}{q}=\tan\frac{\pi}{p} \tan\frac{\pi}{q}\]
Since $p,q>2$, the last equality  holds if and only if 
 either $\alpha=\beta= 1$ or $\alpha=p-1$ and $\beta=q-1$.
\end{proof}

Another result which will be very useful in the analysis of clique-pictures is the following theorem which goes by the name \textit{Spelling Theorem for generalised triangle groups.}
\begin{thm}\cite{hk}\label{SP}
Let $H=\<~x,~y~|~x^p,~y^q,~W(x,y)^r ~\>$ be a generalised triangle group with $W(x,y)=\prod_{i=1}^k x^{\alpha_{_i}} y^{\beta_{_i}}$, $(k>0, 0<\alpha_{_i}<p, 0<\beta_{_i}<q)$. If $V(x,y)=\prod_{i=1}^l x^{\gamma_{_i}} y^{\delta_{_i}}$, $(l>0, 0<\gamma_{_i}<p, 0<\delta_{_i}<q)$ is trivial in $H$, then $l\geq kr$.
\end{thm}
\section{Pictures and Clique-pictures}\label{sec4}
Pictures have been used widely to prove results for one relator groups. In this section we recall only the basic ideas about pictures and clique-pictures as can be found in \cite{hs}. 

\subsection{Pictures}
A \textit{picture} $P$ over $G$ on an oriented surface $\Sigma$ is made up of the following data:
\begin{itemize}
  \item a finite collection of pairwise closed discs in the interior of $\Sigma$ called \textit{vertices}
  
  \item a finite collection of disjoint closed arcs called \textit{edges}, each of which is either: 
  
  \begin{itemize}
  \item a simple closed arc in the interior of $\Sigma$ meeting no vertex of $P$,
  \item a simple arc joining two vertices (possibly same one) on $P$,
  \item a simple arc joining a vertex  to the boundary $\partial \Sigma$ of $\Sigma$,
  \item a simple arc joining $\partial \Sigma$ to $\partial \Sigma$,
  \end{itemize}
  
  \item a collection of \textit{labels} (i.e words in $G_{_1}\cup  G_{_2}$), one for each corner of each \textit{region} (i.e connected component of the complement in $\Sigma$ of union  of vertices and arcs of $P$) at a vertex and one along each component of the intersection of the region with $\partial \Sigma$. For each vertex, the label around it spells out the word $R^{\pm n}$ (up to cyclic permutation) in the clockwise order as a cyclically reduced word in $G_{_1} * G_{_2}$. We call a vertex \textit{positive} or \textit{negative} depending on whether the label around it is $R^{ n}$ or $R^{-n}$ respectively.
\end{itemize}

For us $\Sigma$ will either be the $2$-sphere $S^2$ or $2$-disc $D^2$. A picture on $\Sigma$ is called \textit{spherical} if either $\Sigma=S^2$ or $\Sigma=D^2$ but with no arcs connected to $\partial {D^2}$. If 
$P$ is not spherical, $\partial {D^2}$ is one of the boundary components of a non-simply connected region (provided, of course, that $P$ contains at least one vertex or arc), which is called the \textit{exterior}. All other regions are called \textit{interior}.

\medskip
We shall be interested only in \textit{connected} pictures (to be defined later). This implies that all  interior
regions $\bigtriangleup$ of $P$ are simply-connected i.e topological discs. Just as in the case of vertices, the label around each region  gives a word which is required to be trivial in $G_{_1}$ or $G_{_2}$. Hence it makes sense to talk of $G_{_1}-$regions or $G_{_2}-$regions. Each arc is required to separate a $G_{_1}-$region from a $G_{_2}-$region. This is compatible with the alignment of regions around a vertex, where the labels spell a cyclically reduced word, so must come alternately from $G_{_1}$ and $G_{_2}$. 

\medskip
Likewise  a vertex is called \textit{exterior} if it is possible to join it to the \textit{exterior} region by some arc without intersecting any other arc, and \textit{interior} otherwise. For simplicity we will indeed assume from this point that our $\Sigma$ is either $S^2$ or $D^2$. It follows that reading the label round any \textit{interior} region spells a word which is trivial in $G_{_1}$ or $G_{_2}$. The \textit{boundary label} of $P$ on $D^2$ is a word obtained by reading the \textit{labels} on $\partial D^2$ in an \textit{anticlockwise} direction. This word (which we may be assumed to cyclically reduced in $G_{_1} * G_{_2}$) represents an identity element in $G$. In the case where $P$ is spherical, the \textit{boundary label} is an element in $G_{_1}$ or $G_{_2}$ determined by other labels in the \textit{exterior} region. 

\medskip
Two distinct vertices of a picture are said to  \emph{cancel} along an arc $e$ if they are joined by $e$ and if their labels, read from the endpoints of $e$, are mutually inverse
words in $G_{_1} * G_{_2}$. Such vertices can be removed from a picture via a sequence of \textit{bridge moves} (see Figure \ref{bridge} and  \cite{dh2} for more details), followed by deletion of a \textit{dipole} without changing the boundary label. A \textit{dipole} is a connected spherical picture containing precisely two vertices, does not meet $\partial \Sigma$, and none of its interior regions contain other components of $P$. This gives an alternative picture with the same boundary label and two fewer vertices. 

\begin{figure}[h!]
\includegraphics[scale=0.095]{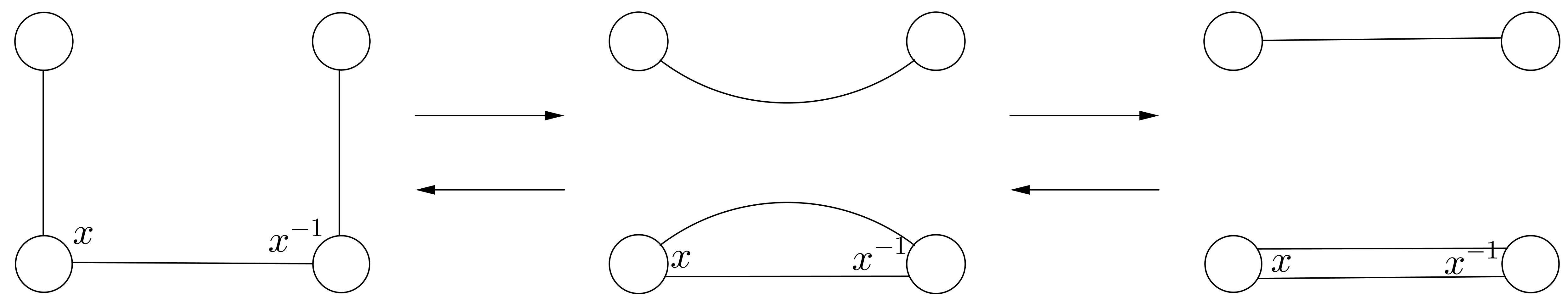}
\caption{\textit{Diagram showing bridge-move.}}
\label{bridge}
\end{figure}
\medskip
We say that a picture $P$ is \textit{reduced} if it cannot be altered by bridge moves to a picture with a pair of cancelling vertices, \textit{minimal} if it is non-empty and has the minimum number of vertices amongst all pictures over $G$, and \textit{efficient} if it has the minimum number of vertices amongst all pictures over $G$ with particular boundary label. Any cyclically reduced word in $G_{_1} * G_{_2}$ representing the identity element of $G$ occurs as the boundary label of some reduced picture on $D^2$. A picture is \textit{connected} if the union of its vertices and arcs is connected. In particular, no arc of a connected picture is a closed arc or joins two points of $\partial \Sigma$, unless the picture consists only of that arc.  

\subsection{Clique-pictures}

\medskip
Clique-pictures appeared in \cite{hs} and are modelled on generalised triangle groups. For the rest of this paper, $$G=\dfrac{(G_{_1} * G_{_2})}{N(R^n)}$$ is a one-relator product induced by the triangle group $$H=\langle x, y | x^p, y^q, R'(x,y)^n\rangle.$$ In other words $R$ is a word in $\lbrace a, UbU^{-1}\rbrace$ for some word $U\in G_{_1} * G_{_2}$ and letters $a$ and $b$ in $G_{_1}\cup G_{_2}$ with orders $p$ and $q$ respectively.

\medskip
If $u$ and $ v$ are two vertices in a picture over $G$ that are joined by an arc $e$, then we may use the endpoints of $e$ as the starting points for reading the labels $L_{_u}$ and $L_{_v}$ of
$u$ and $ v$ respectively. In each case the label is a cyclic permutation of $R'(a, UbU^{-1})^{\pm n}$. We may assume, without loss of generality, that the word $R'(x, y)$ begins with the
letter $x$. Choose a cyclic permutation $R^{*}(x, y)$ of $R'(x, y)^{-1}$ that also starts with $x$.

\medskip
Each of $L_{_u}$ and $L_{_v}^{-1}$ is a cyclic conjugate of $R'(a, UbU^{-1})^n$ or $R^{*}(a, UbU^{-1})^n$, say $L_{_u} = Y Z$, where $ZY = R'(a, UbU^{-1})^n$ or $ZY =R^{*}(a, UbU^{-1})^n$ and 
$L_{_v}^{-1} = Y'Z'$, where $Z'Y' = R'(a, UbU^{-1})^n$ or $Z'Y' =R^{*}(a, UbU^{-1})^n$.

\medskip
We define $u \sim v$ if and only if $\ell(Y') \equiv \ell(Y ) $mod$  l=\ell(aUbU^{-1})$. 
 It follows immediately from
Lemma \ref{refine} 
that $\ell(Y')$ and $\ell(Y )$ are unique modulo  $l$, and so the relation $\sim$ is well-defined. The point of the relation $\sim$ is that, when $u\sim v$, then the $2$-vertex sub-picture consisting of $u$ and $v$, joined by $e$ and any arcs \textit{parallel} (see Remark \ref{rmk} for definition) to $e$, has boundary label a word in $\lbrace a, UbU^{-1}\rbrace$, after cyclic reduction and cyclic permutation. (Indeed, the cyclic reduction of the label can be achieved by performing bridge moves to make the number of edges parallel to $e$ be a multiple of  $l/2$.) Now let $\approx$ denote the transitive, reflexive closure of $\sim$. Then $\approx$ is an equivalence relation on vertices. After a sequence of bridge moves, we may assume that arcs joining equivalent vertices do so in parallel classes each containing a multiple of $l/2$ arcs. Define a \textit{clique} to be the sub-picture consisting of any $\approx$-equivalence class of vertices, together with all arcs between vertices in that $\approx$-class (assumed to occur in parallel classes of multiples of $l/2$ arcs), and all regions that are enclosed entirely by such arcs.

\begin{defn}
Let $G$ be a one-relator product induced from a generalised triangle
group as above, and let $P$ be a picture on a surface $\Sigma$, such that every clique of $P$
is simply-connected. Then the \textit{clique-quotient} of $P$ is the picture  formed from  $P$
by contracting each clique to a point, and regarding it as a vertex.
A \textit{clique-picture} \textbf{P} over $G$ is the clique-quotient of some (reduced) picture over $G$.
The label of a vertex in a clique-picture is called a \textit{clique-label}.
\end{defn}

The process of joining two vertices of $P$ or two cliques of \textbf{P} to form a single {clique} is called \textit{amalgamation}. 
 (Here we also include the possibility of amalgamating a clique with itself. By this we mean adding arcs from $v$ to $v$ and/or regions to an existing clique
$v$, which could alter some properties of the clique such as simple-connectivity.)
If it is possible to amalgamate two cliques (possibly after doing bridge-moves), we say that \textbf{P} is not \textit{reduced}, and \textit{reduced} otherwise. The minimality and efficiency conditions carry over from pictures.

\begin{rem}\label{rmk}
Let $\Gamma$ be $P$ or $\textbf{P}$. Two arcs of $\Gamma$ are said to be \textit{parallel} if they are the only two arcs in the boundary of some simply-connected region $\bigtriangleup$ of $\Gamma$. We will also use the term \textit{parallel} to denote the equivalence relation generated by this relation, and refer to any of the corresponding equivalence classes as a \textit{class of $\omega$ parallel arcs} or \textit{$\omega$-zone}. Given a \textit{$\omega$-zone} joining vertices $u$ and $v$ of $\Gamma$, consider the $\omega- 1$ two-sided regions separating these arcs. Each such region has a corner label $x_{_u}$ at $u$ and a corner label $x_{_v}$ at $v$, and the picture axioms imply that $x_{_u}x_{_v} = 1$ in $G_{_1}$ or $G_{_2}$. The $\omega -1$ corner labels at $v$ spell a cyclic subword $s$ of length $\omega-1$ of the label of $v$. Similarly the corner labels at $u$ spell out a cyclic subword $t$ of length $\omega -1$. Moreover, $s=t^{-1}$. If we assume that $\Gamma$ is reduced, then $u$ and $v$ do not cancel. Hence the cyclic permutations of the labels at $v$ and $u$ of which $s$ and $t$ are initial segments respectively are not equal. Hence $t$ and $s$ are \textit{pieces}. 
\end{rem}

As in graphs, the \textit{degree} of a vertex in $\Gamma$ is the number of \textit{zones} incident on it. For a region, the \textit{degree} is the number corners it has. We say that a vertex $v$ of $\Gamma$ satisfies the local $C(m)$ condition if it is joined to at least $m$ \textit{zones}. We say that $\Gamma$ satisfies $C(m)$ if every interior vertex satisfies local $C(m)$.

\section{Preliminary results}\label{sec5}
In this section we obtain some preliminary results about clique-pictures. One advantage  clique-pictures has over ordinary pictures is that 
 some cyclic permutation of the inverse of any clique-label can also be interpreted as a clique-label. Thus we may regard any clique as having either possible orientation, as convenient.  We make the convention
that all our cliques
have the same (positive or clockwise) orientation. 

Throughout this paper we shall assume that our clique-picture is minimal. Note that up to cyclic permutation the clique-label of a clique $u$
has the form
\begin{equation}\label{eqn}
\mathrm{label}(u)=\prod_{i=1}^k a^{\alpha_{_i}} U b^{\beta_{_i}} U^{-1}
\end{equation}

for $0<\alpha_{_i}<p$ and $0<\beta_{_i}<q$.  Denote our clique-picture by $\Gamma$ and let $v$ be a clique of $\Gamma$. Take a cyclic permutation $c(k)$ of  the label of $v$ of the form (\ref{eqn}) and express it as $$c(k)= z_{ 0} z_{_1} \cdots z_{_{kl-1}} $$ where $l=\ell(a^{\alpha_{_i}} U b^{\beta_{_i}} U^{-1})$.

\medskip
 We call a letter $z_j$ of a clique-label
$$\mathrm{label}(u)=\prod_{i=1}^k a^{\alpha_{_i}} U b^{\beta_{_i}} U^{-1}= z_{ 0} z_{_1} \cdots z_{_{kl-1}} $$
{\em special} if $j\equiv 0$ mod $l/2$.  Note that every special letter is equal to a power of $a$ or of $b$.

\medskip
Let $\Omega:=(G_{_1}\sqcup G_{_2})\setminus\{1\}$.  Define $\sim$ to be the smallest equivalence relation on $\Omega$ with the property that $a^\alpha\sim a$ for all $\alpha$ such that $a^\alpha\neq 1$ and $b^\beta\sim b$ for all $\beta$ such that $b^\beta\neq 1$. Note that the natural involution $x\mapsto x^{-1}$ on $\Omega$ descends to an involution on $\Omega/\sim$, under which the $\sim$-classes of $a$ and $b$ are fixed points.

\medskip
We  will sometimes work in the free monoid $(\Omega/\sim)^*$. In particular clique-labels are periodic  in $(\Omega/\sim)^*$ with period $l$.

\medskip
\paragraph{\textbf{Notations}}
Let $v$ be a clique of degree $k$. This means that there are $k$ zones incident at $v$, say $Z_{_1}, Z_{_2}, \ldots, Z_{_k}$ labelled consecutively in clockwise order around $v$ as shown in Figure \ref{zns}.
\begin{figure}[h!]
\centering
\includegraphics[width=.3\linewidth]{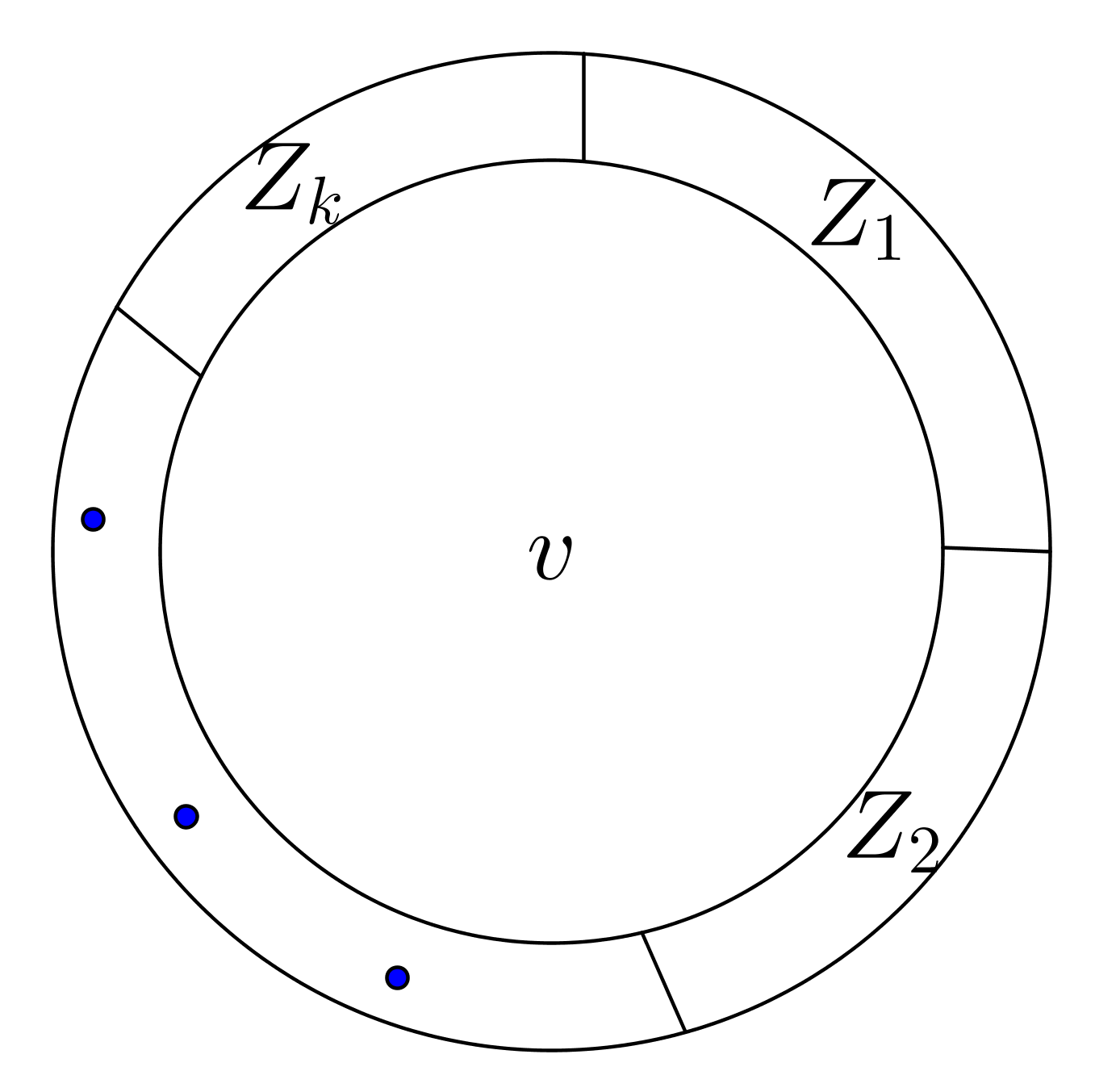}
 \caption{\textit{Zones}} \label{zns}
\end{figure} 

Recall that each zone $Z_{_i}$ is a class of parallel arcs. The number of arcs in $Z_{_i}$  is denoted by $\omega_{_i}$. If $Z_{_i}$ connects cliques $u$ and
$v$ (possibly $u=v$), then $Z_{_i}$ determines cyclic
subwords $s_{_i},t_{_i}$ of length $\omega_{_i}-1$ of the clique-labels 
of $u,v$ respectively, such that $s_{_i}\equiv t_{_i}^{-1}$ in $G_{_1}*G_{_2}$

A zone $Z_i$ is said to be {\em large} if $\omega_{_i}$

\section{Theorem \ref{thmm1}}\label{sec6}

In this section we prove Theorem \ref{thmm1}.  
We first note another consequence of Lemma \ref{nz2}.

\medskip
We will use the following generalisation of the concept of
periodic word, as applied to cyclic subwords of 
 $$\mathrm{label}(v)=z_0z_1\cdots z_{nl-1}=
\prod_{k=1}^n [a^{\alpha(k)}Ub^{\beta(k)}U^{-1}].$$
We say that a cyclic subword $W=z_j\cdots z_k$ (subscripts modulo $nl$) of $\mathrm{label}(v)$ is {\em virtually periodic} with {\em virtual period} $\mu$ if, for each $i\in\{j,j+1,\dots,k-\mu\}$, one of the following happens:

\begin{enumerate}
\item $z_i=z_{i+\mu}$;
\item a special letter $z_d=a^\psi$ belongs to $W$, for
some $d\equiv 0$ mod $l$, $i\equiv d$ mod $\mu$, and each of $z_i$, $z_{i+\mu}$ is equal to a power of $a$;
\item a special letter $z_d=b^\psi$ belongs to $W$, for
some $d\equiv l/2$ mod $l$, $i\equiv d$ mod $\mu$, and each of $z_i$, $z_{i+\mu}$ is equal to a power of $b$;
\item $a$ and $b$ have a common root $c$ in $G_{_1}$ or $G_{_2}$,
a special letter $z_d=c^\psi$ belongs to $W$, for
some $d\equiv 0$ mod $l/2$, $i\equiv d$ mod $\mu$, and each of $z_i$, $z_{i+\mu}$ is equal to a power of $c$.
\end{enumerate}

Recall that the pair $(a,b)$ is assumed to be admissible.
If $a$ and $b$ have a common power in $G_{_1}$ and $G_{_2}$, then
they have a common root, and in that case the second and third possibilities in the above definition are subsumed in the fourth.  Otherwise the fourth possibility cannot occur.

By definition, the clique-label $\mathrm{label}(v)$ itself is
virtually periodic of virtual period $l$.  Other examples of virtually periodic words arise from 
zones incident at $v$.

\begin{lem}\label{largezone}
Suppose that $Z_i$ is a
 zone incident at $v$.  Then
there is a positive integer $\mu\le l/2$ and a cyclic subword
$s_i^+$ of $\mathrm{label}(v)$ of length $\omega_i+\mu-1$
and virtual period $\mu$, such that $s_i$ is either an initial or a terminal segment of $s_i^+$.
\end{lem}

\begin{proof}
Let $s_i$ be the cyclic subword $z_jz_{j+1}\cdots z_k$ of $$\mathrm{label}(v)=z_0z_1\cdots z_{nl-1}=
\prod_{k=1}^n [a^{\alpha(k)}Ub^{\beta(k)}U^{-1}].$$

The zone $Z_i$ links $v$ to an adjacent clique $u$
and identifies $s_i$ with $t_i^{-1}$ for some cyclic
subword $t_i$ of $\mathrm{label}(u)$. Thus $t_i^{-1}$
is a cyclic subword of $\mathrm{label}(u)^{-1}$.  Write
$t_i^{-1}=y_{j'}y_{j'+1}\cdots y_{k'}$ where
$$\mathrm{label}(u)^{-1}=y_0y_1\cdots y_{ml-1}=
\prod_{k=1}^m [a^{\gamma(k)}Ub^{\delta(k)}U^{-1}].$$

Since $s_i\equiv t_i^{-1}$, then in particular
$k'-j'\equiv\ell(t_i)-1=\ell(s_i)-1\equiv k-j$ mod $l$.  If $j\equiv j'$
mod $l$, then we may amalgamate the cliques $u$ and $v$, contrary to hypothesis.  Hence there are
integers $n$ and $\mu_i$ such that
$0\le n\le m$, $0<\mu_i\le l/2$ such that
$j'=j+nl\pm\mu_i$.  Define
$$s_i^+=\left\{\begin{array}{rcl}
z_{j-\mu_i}\cdots z_k~~&\mathrm{if}~~& j'=j+nl-\mu_i\\
z_j\cdots z_{k+\mu_i}~~&\mathrm{if}~~&
j'=j+nl+\mu_i\end{array}\right.$$

In the first case, $s_i$ is a terminal segment of $s_i^+$, while the initial segment of the same length
agrees with $t_i^{-1}\equiv s_i$, except possibly
at special letters $z_d$ ($d\equiv 0$ mod $l/2$)
which may be a different power of $a$ (or of $b$) than
the corresponding letter of $s_i$.  It follows that
$s_i^+$ is virtually periodic of virtual period $\mu_i$, as claimed.

The second case is entirely analogous, except that $s_i$ is an initial rather than a terminal segment
of $s_i^+$.
\end{proof}

We need to analyse the interaction of virtually periodic subwords of $\mathrm{label}(v)$ obtained by applying Lemma \ref{largezone} to two adjacent large zones at $v$.  To do this we will use the following analogue of Corollary \ref{ctm}.

\begin{lem}\label{twozones}
Suppose that the cyclic subword $W=z_j\cdots z_k$ (subscripts modulo $nl$) of 
$$ \mathrm{label}(v)=z_0z_1\cdots z_{nl-1}=
\prod_{k=1}^n [a^{\alpha(k)}Ub^{\beta(k)}U^{-1}].$$
is the union of a virtually
periodic segment
 $W_1$
of virtual period $\mu$ and
 a virtually
periodic
segment $W_2$
of virtual period $\nu$.  Let $\gamma=\gcd(\mu,\nu)$.  If the intersection 
of these segments has length at least $\mu+\nu-\gamma$, then $W$ is virtually periodic of virtual period $\gamma$.
\end{lem}

\begin{proof}
Let $i,m$ be such that $z_i$ and $z_{i+m\gamma}$
are letters of $W$.  Then we claim there is a finite chain of  subscripts
$i(0),i(1),\dots,i(N)$ with $i(0)=i$ and $i(N)=i+ m\gamma$ such  that, for each $t$ either $|i(t)-i(t+1)|=\mu$ and $z_{i(t)}$ and $z_{i(t+1)}$ are letters of $W_1$, or $|i(t)-i(t+1)|=\nu$ and $z_{i(t)}$ and $z_{i(t+1)}$ are letters of $W_2$.
Certainly each letter in $W_1$ (resp. $W_2$) is linked to some letter in $W_0:=W_1\cap W_2$ by such a chain, since $\ell(W_0)\ge\max(\mu,\nu)$, so it suffices to prove the claim when $z_i,z_{i+m\gamma}$ are letters of $W_0$.  Write
$m\gamma=\alpha\mu+\beta\nu$ where $\alpha,\beta\in\mathbb{Z}$, and argue by induction on $|\alpha|+|\beta|$.  Without loss of generality, assume that $\alpha>0$.
If $z_{i+\mu}$ is a letter of $W_0$, then the result follows by applying the inductive hypothesis to $z_{i+\mu},z_{i+m\gamma}$.  Otherwise, $\beta<0$ and $z_{i-\nu}$ is a letter of $W_0$, so we may apply the inductive hypothesis to $z_{i-\nu},z_{i+m\gamma}$.  This proves the claim.

Now we take $m=1$ in the above, and prove that at least one
of the alternative conditions for virtual periodicity holds.

If $z_{i(t)}=z_{i(t+1)}$ for all $n$ then $z_i=z_{i+\gamma}$.

Suppose next that $z_{i(t)}\ne z_{i(tt+1)}$ for at least one value of
$t$, and $a,b$ have a common root $c$. Then there
is a special letter $z_d$ in $W$, $i\equiv i(t) \equiv d$ mod $\mu$ or mod $\nu$ (and hence in either case mod $\gamma$). Moreover for each $t$ either $z_{i(t)}=z_{i(t+1)}$ or each of $z_{i(t)},z_{i(t+1)}$ is a power of $c$.  Since 
$z_{i(t)}\ne z_{i(t+1)}$ for at least one value of
$t$, it follows that each $z_{i(t)}$ is a power of $c$.  In particular $z_i$ and $z_{i+\gamma}$ are both powers of $c$.

Finally, suppose that $z_{i(t)}\ne z_{i(t+1)}$ for at least one value of
$t$, and $a,b$ have no common root.  By admissibility, $a,b$ also have no common non-trivial power.  As above, $W$ contains a special letter $z_d=a^\psi$ or $z_d=b^\psi$, and $i\equiv d$ mod $\gamma$.  Consider the least $t$ for which $z_{i(t)}\ne z_{i(t+1)}$.  Then either $z_{i(t)}$ and $z_{i(t+1)}$ are both powers of $a$ or both powers of $b$.  Assume the former.  Then $z_i=z_{i(0)}=z_{i(1)}=\cdots =z_{i(t)}$ are also powers of $a$.  We claim that $z_{i(t+1)}$, \dots $z_{i(N)}=z_{i+\gamma}$ are also powers of $a$, which will complete the proof.  Suppose by way of contradiction that this is not true.  Then there is an $m$ for which $z_{i(m)}$ is a power of $a$ and $z_{i(m+1)}$ is not a power of $a$.  By the definition of virtual periodicity, it follows that both $z_{i(m)}$  and $z_{i(m+1)}$ must be powers of $b$.  But then $z_{i(m)}$ is simultaneously a power of $a$ and of $b$, contrary to the admissibility hypothesis.
\end{proof}

\begin{cor}\label{refine2}
If a clique-label has virtual period $\mu<l$, then a
refinement is possible.
\end{cor}

\begin{proof}
By Lemma \ref{twozones} the clique-label 
$$\mathrm{label}(v)=\prod_{j=1}^n [a^{\alpha(j)}Ub^{\beta(j)}U^{-1}]$$
has virtual
period $\gcd(\mu,l)|l$, so without loss of generality $\mu|l$.  Let $V=z_1\cdots z_{\frac\mu2}$ and let $s=l/\mu$.  Then by definition of virtual periodicity
and by the admissibility hypothesis one of the following is true:
\begin{enumerate}
\item $s$ is odd and $$\mathrm{label}(v)=\prod_{j=1}^{sn} [a^{\gamma(j)}Vb^{\delta(j)}V^{-1}]$$ for some $\gamma(j),\delta(j)$.
\item $s$ is even, $a,b$ have a common root $c$, and $$\mathrm{label}(v)=\prod_{j=1}^{sn} [c^{\gamma(j)}VxV^{-1}]$$
for some letter $x$ of order $2$ and some $\gamma(j)$.
\end{enumerate}
In either case, we have a refinement.
\end{proof}

\begin{cor}\label{lz}
Suppose that $v$ is a clique in a minimal clique-picture
over $G$ satisfying Hypothesis A.  Suppose also that 
the generalised triangle group description of $G$ has no refinement.
Then the length of any zone incident at $v$ is strictly less than $l$.
\end{cor}

\begin{proof}
The $i$'th zone $Z_i$ contains $\omega_i$ arcs. 
Assume that $\omega_i\ge l>l/2$.  Then by Lemma
\ref{largezone}, 
the word $s_i^+$ has length $\omega_i+\mu_i-1$, which is
strictly greater than $l+\mu_i-\gcd(l,\mu_i)$
since $\gcd(l,\mu_i)$ is even. Moreover, $s_i^+$ is virtually periodic of virtual period $\mu_i$.  But $s_i^+$ is
a cyclic subword of the clique-label $\mathrm{label}(v)$
which is virtually periodic with virtual period $l$.

By  Lemma \ref{twozones} it follows that $\mathrm{label}(v)$
has virtual period $\gcd(l,\mu_i)\le l/2$.
But by Corollary \ref{refine2} this leads to a refinement of our generalised triangle group description of $G$, contrary to hypothesis. 

\end{proof}

Assuming Hypothesis A, we have a minimal clique-picture over
a one-relator product $G=(G_{_1}*G_{_2})/N(R(a,UbU^{-1})^n)$ with
$n\ge 2$ and $\ell(R)\ge 4$ as a word in $\<a\>*\<UbU^{-1}\>$.
Any clique-label has the form
$$ \mathrm{label}(u)=\prod_{j=1}^k [a^{\alpha(j)}Ub^{\beta(j)}U^{-1}]$$
By the Spelling Theorem \ref{SP} we must have $k\ge 2nl$
where $l=\ell(aUbU^{-1})$.  But by Corollary \ref{lz} 
each zone has fewer than $l$ arcs, so all cliques
have degree at least $2n+1$.  To prove the theorem, we assume that $n=2$ and that $v$ is a clique of degree $5$, with zones
$Z_1,\dots Z_5$ of sizes $\omega_1,\dots \omega_5$ respectively,
in cyclic order around $v$, and aim to derive a contradiction.

The key tool in the proof of Theorem \ref{thmm1} is the following.

\begin{lem}\label{twoshort}
For each $i=1,\dots,5$, one of the following holds:
\begin{enumerate}
\item $\omega_i+\omega_{i-1}<3l/2$;
\item $\omega_i+\omega_{i+1}<3l/2$.
\end{enumerate}
\end{lem}

\begin{proof}

By Lemma \ref{largezone}, $s_i$ is either an initial segment or a terminal segment of a subword $s_i^+$ of $\mathrm{label}(v)$ of length $\omega_i+\mu_i-1$ and virtual period $\mu_i$, where
$0<\mu_i\le l/2$.  We will assume that $s_i$ is an initial segment of $s_i^+$ and show that $\omega_i+\omega_{i+1}<3l/2$.  (An entirely analogous argument shows that, if $s_i$ is a terminal segment of $s_i^+$, then $\omega_{i-1}+\omega_i<3l/2$.)

 Now apply Lemma \ref{largezone}
to $s_{i+1}$: $s_{i+1}$ is either an initial segment or a terminal segment of a cyclic
subword $s_{i+1}^+$ of $\mathrm{label}(v)$ of length $\omega_{i+1}+\mu_{i+1}$ and virtual period $\mu_{i+1}$, where $0<\mu_{i+1}\le l/2$.  Our argument splits into two cases, depending on whether $s_{i+1}$ is an initial or terminal segment of $s_{i+1}^+$.

\medskip\noindent{\bf Case 1.} $s_{i+1}$ is a terminal segment of $s_{i+1}^+$.

Consider the cyclic subword $W:=s_i z_{s(i+1)} s_{i+1}$ of $\mathrm{label}(v)$.  Since $\mu_i\le l/2<\omega_{i+1}$, the virtually periodic subword $s_i^+$ is an initial segment of $W$.
Similarly, $s_{i+1}^+$ is a terminal segment of $W$.
These segments intersect in a segment of length
$\mu_i+\mu_{i+1}-1>\mu_i+\mu_{i+1}-\lambda$, where $\lambda=\gcd(\mu_i,\mu_{i+1})$, so by Lemma \ref{twozones} $W$ is virtually periodic with period $\lambda$.

Now recall that $W$ is a cyclic subword of $\mathrm{label}(v)$, which is virtually periodic with
virtual period $l$.  If $W$ has length greater
than $l+\lambda-2$, then $\mathrm{label}(v)$ has
virtual period $\gcd(l,\lambda)<l$, by
another application of Lemma \ref{twozones}.

But by Corollary \ref{refine2} this leads to a refinement of our generalised triangle group description of $G$, contrary to hypothesis.  Thus
$$\omega_i+\omega_{i+1}=1+\ell(W)\le  l+\lambda-1<3l/2.$$

\medskip\noindent{\bf Case 2.} $s_{i+1}$ is an initial segment of $s_{i+1}^+$.

Let $\bar{s}_i$, $\bar{s}_i^+$ denote the cyclic
subwords of $\mathrm{label}(v)$ that begin with the
letter exactly $l$ places after the first letter
of $s_i$.  By the virtual periodicity of $\mathrm{label}(v)$ it follows that $\bar{s}_i^+$
is also virtually periodic, of virtual period $\mu_i$.  Moreover, $\bar{s}_i^+$ has length $\omega_i+\mu_i-1$ and has $\bar{s}_i$ as an initial
segment.

By construction, the union of the subwords $s_{i+1}$ and $\bar{s}_i$ of $\mathrm{label}(v)$ has length
$l-1$.  Let $W$ be the union
of the subwords $s_{i+1}^+$ and $\bar{s}_i^+$ of $\mathrm{label}(v)$.  Then $W$ has length at least
$l+\mu_i-1$.  Arguing as in Case 1, we obtain a
refinement of the generalised triangle group description of $G$, contrary to hypothesis, if
$s_{i+1}^+$ and $\bar{s}_i^+$ intersect in a segment of length $\mu_i+\mu_{i+1}-\gcd(\mu_i,\mu_{i+1})$ or
greater.  So we may assume that this does not happen.

In particular, $\mu_{i+1}<l-\omega_{i+1}+\mu_i$, for otherwise $\bar{s}_i^+$ is a subword of $s_{i+1}^+$, of length $\omega_i+\mu_i-1>l/2+\mu_i-1>\mu_{i+1}+\mu_i-\gcd(\mu_i,\mu_{i+1})$.  Hence $\bar{s}_i^+$ is a terminal segment of $W$, and  the intersection of $s_{i+1}^+$
and $\bar{s}_i^+$ has length precisely
$\omega_{i+1}+\mu_{i+1}+\omega_i-l-1$.

Thus $$\omega_i+\omega_{i+1}<\mu_i+\mu_{i+1}-\gcd(\mu_i,\mu_{i+1})+l+1-\mu_{i+1}< l+\mu_i \le 3l/2$$
as claimed.

\end{proof}

Using Lemma \ref{twoshort}, we complete the proof of Theorem \ref{thmm1} as follows.
Renumbering the zones if necessary, we may assume by Lemma \ref{twoshort} that $\omega_1+\omega_2<3l/2$.  Applying
Lemma \ref{twoshort} again, with $i=4$, either $\omega_3+\omega_4<3l/2$ or $\omega_4+\omega_5<3l/2$.
In the first case $\omega_5>l$; in the second $\omega_3>l$.  Either of these is a contradiction.

\section{Theorem \ref{thmm2} }\label{sec7}

In this section we prove Theorem \ref{thmm2}.
In order to do so we will need a number of lemmas that
are particular to the situation of Theorem \ref{thmm2}, which
we also collect together in this section.

Recall that $G=(G_{_1}*G_{_2})/N(R(a,UbU^{-1})^n)$ where $n\ge 2$
and the relation $R(a,UbU^{-1})^n$ contains no letters
of order $2$. 
 We assume that $v$ is a clique of degree less than $6$ in a clique picture, joined to neighbouring cliques
$u_i$ by zones $Z_i$.
 The first step in our proof is designed to
further restrict the form of $R$ and hence 
also of $\label(v)$.

\begin{lem}\label{od2}
For any zone $Z_i$, if $s(i)+j\equiv t(i)+m$ mod $l$ where $1\leq j,m<\omega_i$, then $s_i$ has an element of order $2$ in $G_{_1}$ or $G_{_2}$.
\end{lem}

\begin{proof}
Suppose by contradiction that $s(i)+j\equiv t(i)+m$ mod $l$ where $1\leq j,m<\omega_i$. 
 Recall that $s_i$ is a cyclic subword of
$$\mathrm{label}(v)=z_0z_1\cdots z_{nl-1}=
\prod_{k=1}^n [a^{\alpha(k)}Ub^{\beta(k)}U^{-1}].$$
Write
$s_i=z_{s(i)+1}\cdots z_{s(i+1)-1}$.
 Similarly, $t_i$ is a cyclic subword of
$$\mathrm{label}(u_i)=y_0y_1\cdots y_{ql-1}=
\prod_{k=1}^q [a^{\gamma(k)}Ub^{\delta(k)}U^{-1}].$$
Write
$t_i=y_{t(i)+1}\cdots y_{t(i+1)-1}$.
By hypothesis, $y_{t(i)+m}=z_{s(i)+m'}^{-1}$ for some $m'$
with $1\le m'<\omega_i$.  In particular $s(i)+j\equiv t(i)+m\equiv s(i)+m'$ mod $2$, since $y_{t(i)+m}$ and $z_{s(i)+m'}$ belong to the same free factor.  Thus $j+m'$ is even.  Moreover,
$z_{s(i)+(j+m')/2}=y_{t(i)+(j+m')/2}^{-1}$.  If
$z_{s(i)+(j+m')/2}$ is a special letter, then so is
$y_{t(i)+(j+m')/2}^{-1}$.  But in this case an amalgamation is possible, contrary to hypothesis.  Otherwise
$z_{s(i)+(j+m')/2}=y_{t(i)+(j+m')/2}$, so $z_{s(i)+(j+m')/2}$
has order $2$ in $G_{_1}$ or $G_{_2}$.
\end{proof}

\begin{lem}\label{od3weak}
If no letter of $R(a,UbU^{-1})$ has order $2$, then
 there is
no zone  $Z_i$ with $\omega_i > l/2$.
\end{lem}

\begin{proof}
Suppose that $\omega_i\ge l/2$ for some zone $Z_i$.  Write
$s_i=z_{s(i)+1}\cdots z_{s(i+1)-1}$ and 
$t_i=y_{t(i)+1}\cdots y_{t(i+1)-1}$.  By Lemma \ref{od2}
we may assume that $s(i)+j\not\equiv t(i)+m$ mod $l$ for all $j,m\in\{1,2,\dots,\omega_i-1\}$, so in particular $\omega_i\le l/2+1$.  Moreover, if $\omega_i=l/2+1$
then we must have $t(i)+1\equiv s(i+1)$ mod $l$.  But
since $z_{s(i+1)-1}=y_{t(i)+1}^{-1}$ belongs to the same free factor as $z_{t(i)+1}=z_{s(i+1)}$,  this is a contradiction.
Hence $\omega_i\le l/2$.
\end{proof}

 The clique $v$ fails to satisfy the $C(6)$ property, so by Lemma \ref{od3weak} its label has length at most $5l/2$.  But this length is a multiple of $l$, so at most $2l$.
By Theorem \ref{SP} and  Proposition \ref{prop1}
we can assume that $R=aU b U^{-1}$ (up to cyclic permutation) and that the label of $v$ is $\mathrm{label}(v)=(a U b U^{-1})^{\pm2}$ (up to cyclic permutation). Without loss of generality we will assume throughout that this label is $\mathrm{label}(v)=(a U b U^{-1})^{2}$ and the letters $a,b$ will mean the corresponding special letters.

This remark enables us to strengthen Lemma \ref{od3weak} as follows.

\begin{lem}\label{od3}
If no letter of $R(a,UbU^{-1})$ has order $2$, then
 there is
no zone  $Z_i$ with $\omega_i\ge l/2$.
\end{lem}

\begin{proof}
 By Lemma \ref{od3weak} 
we are reduced to the case where $\omega_i=l/2$.
By Lemma \ref{od2}
, we must have $t(i)\in\{s(i+1)-1,s(i+1),s(i+1)+1\}$ mod $l$.  The first possibility
leads to a contradiction 
as in the proof of Lemma \ref{od3weak}.  The third possibility
also leads to a contradiction for similar reasons, since
it implies that $t(i+1)\equiv s(i)+1$ mod $l$.
Hence we may assume that $t(i)\equiv s(i+1)$ mod $l$
and hence that $s(i)\equiv s(i+1)+l/2\equiv t(i)+l/2\equiv t(i+1)$ mod $l$.

If $z_{s(i)}$ is special, then so is $y_{t(i+1)}$, and we may
amalgamate cliques, contrary to hypothesis.  Hence $z_{s(i)}$
is not special.  In other words $s(i)\not\equiv 0$ mod $l/2$.   Thus $s_i$ contains precisely one special letter
-- say $a$  without loss of generality.  Hence also $t$ contains precisely one special letter, which is necessarily
a power of $b$.  Thus $aUb^\psi U^{-1}$ is a proper cyclic permutation of $z_{s(i)}s_iz_{s(i+1)}t_i\equiv z_{s(i)}s_iz_{s(i+1)}s_i^{-1}$ for some $\psi$.  Applying
Lemma \ref{nz2}, we see that 
$$aUb^\psi U^{-1}=\prod_{k=1}^s [a^{\alpha(k)}Vd^{\beta(k)}V^{-1}]$$
for some word $V$, some $\alpha(k),\beta(k)=\pm1$ and some
$d\in\{b^\psi,z_{s(i)},z_{s(i+1)}\}$, where $s>1$.

Moreover, from the proof of Lemma \ref{nz2} we see that we may take $d=b^\psi$ if $s$ is odd, while if $s$ is even then $a=b^\psi$.  In either case, $\<a\>*\<UbU^{-1}\>$ is a
proper subgroup of $\<a\>*\<VbV^{-1}\>$ or $\<a\>*\<VdV^{-1}\>$, giving a refinement of the generalised triangle group description of $G$.  This contradiction completes the proof.

\end{proof}

\begin{rem}\label{remn}
From   Lemma
\ref{od3},  our interior clique $v$ which fails the $C(6)$ condition must have exactly five zones. Let these zones be $Z_{_1}, Z_{_2}, \ldots , Z_{_5}$ listed consecutively in clockwise order.
\end{rem}

\medskip
\begin{prop}
There are exactly three or four of the zones 
 $Z_{_1}, Z_{_2}, \ldots , Z_{_5}$
containing a special letter.
\end{prop}
\begin{proof}

 By assumption $$\mathrm{label}(v)=(aUbU^{-1})^2=z_0z_1\cdots z_{2l-1}$$
so there are precisely four special letters in $\mathrm{label}(v)$, namely $z_0=z_l=a$ and $z_{l/2}=z_{3l/2}=b$.  By Lemma \ref{od3weak} no zone can contain more than one special letter, so it suffices to show that at most one of the special letters is not contained in a zone.  Suppose by way of contradiction that $z_0$ and $z_l$
are not contained in zones.  By Lemma \ref{od3}, each of the subwords $z_1\cdots z_{l-1}$ and $z_{l+1}\cdots z_{2l-1}$
must contain at least three zones, contradicting our assumption that there are only five zones in total.
A similar contradiction arises if $z_0$ and $z_{l/2}$ are not contained in zones: $z_1\cdots z_{l/2-1}$ and $z_{l/2+1}\cdots z_{2l-1}$ must contain at least two and four zones respectively. By symmetry, any other combination of two special letters not contained in zones also contradicts our
underlying hypotheses, hence the result.

\end{proof}

Therefore  we have exactly two possible configurations as depicted in the Figure \ref{fig:fig} below.

\begin{figure}[h!]
\centering
\begin{subfigure}{0.50\textwidth}
  \centering  
  \includegraphics[width=.8\linewidth]{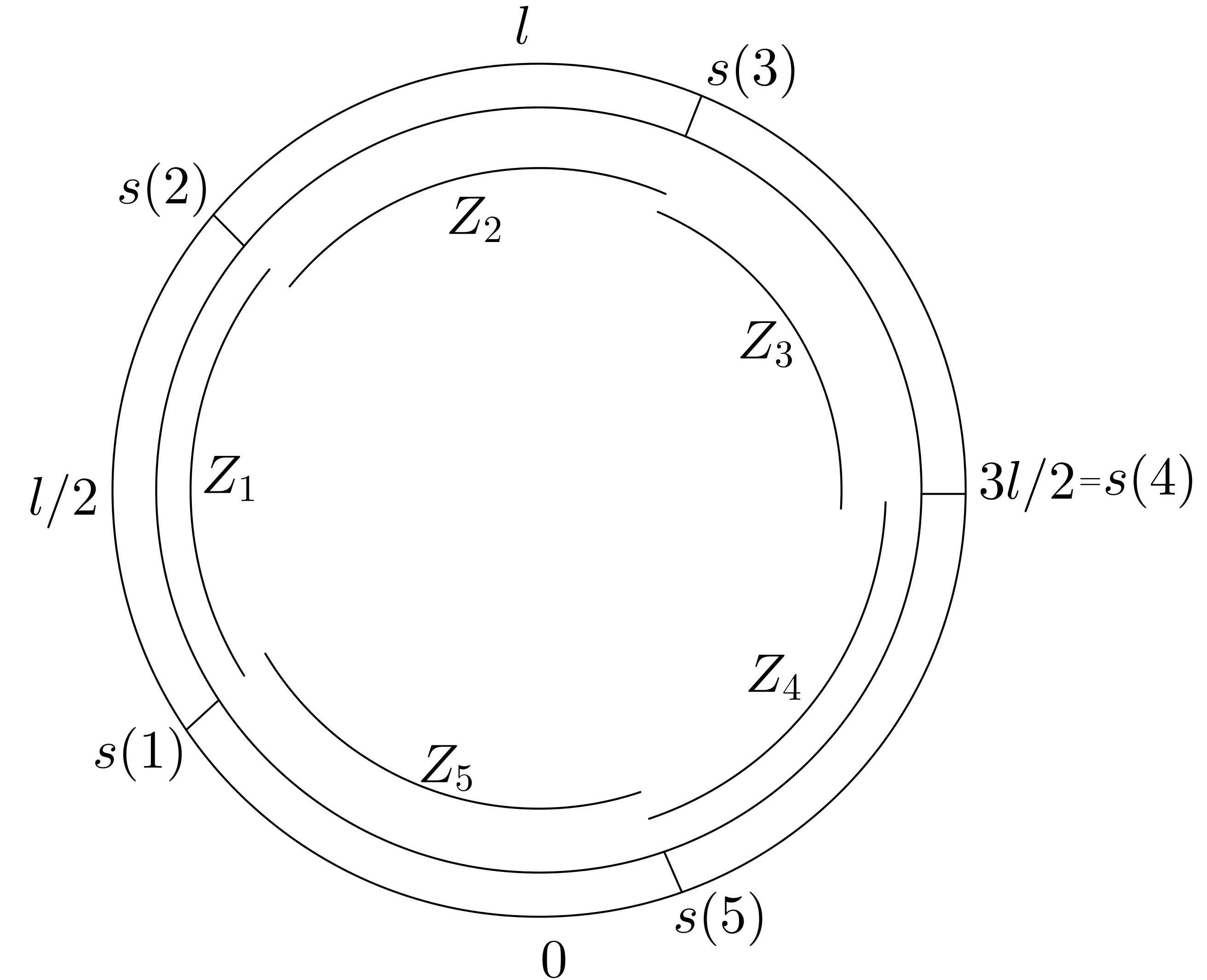}
  \caption{\textit{Exactly three zones containing a special letter.}} \label{part_a}
\end{subfigure}
\begin{subfigure}{0.45\textwidth}
  \centering
  \includegraphics[width=.8\linewidth]{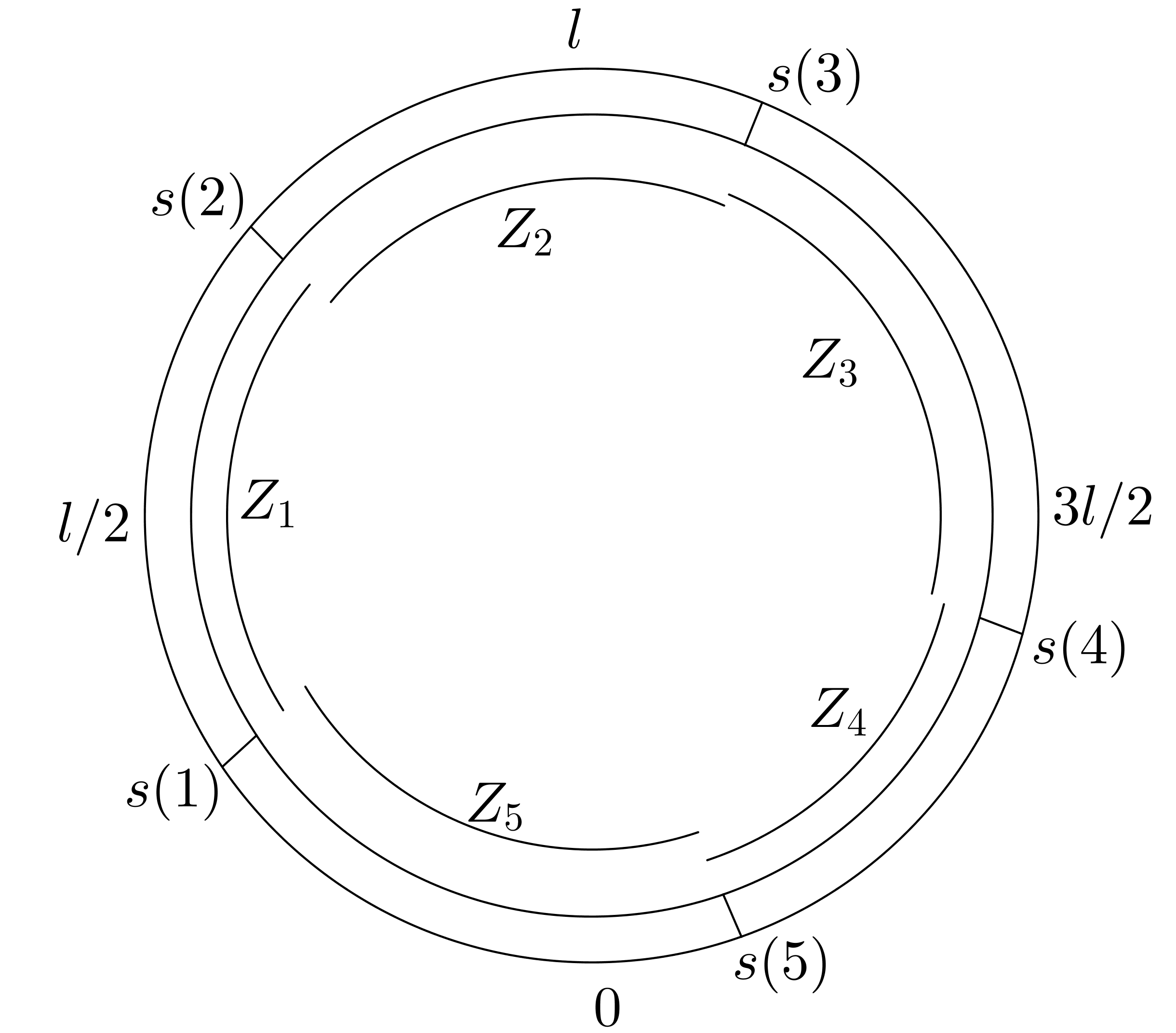}
  \caption{\textit{Exactly four zones containing a special letter.}} \label{part_b}
\end{subfigure}
\caption{\textit{Diagram showing the arrangement (in clockwise order) of five zones around an interior vertex.}} \label{fig:fig}
\end{figure} 
\begin{rem}\label{ine}
For each $z_{s(i)}$, we have  $s(i)\neq s(j)~ mod~ l/2$ for $i\neq j$ since  no $\omega_i\geq l/2$.
\end{rem}
Remark \ref{ine} gives the inequalities 
\begin{equation}\label{isteq1}
0<s(3)<s(1)<l/2=s(4)<s(2)<s(5)\leq l-1
\end{equation}
and 
\begin{equation}\label{isteq2}
0<s(3)<s(1)<l/2<s(4)<s(2)<s(5)\leq l-1
\end{equation}

corresponding to Figure \ref{part_a} and Figure \ref{part_b} respectively. (Note that $s(i)$ in the inequalities is the modulo $l$ equivalent. The actual value can be read from the Figure \ref{fig:fig}. For example the actual value for $s(4)$ in Figure \ref{part_a} is $\frac{3l}{2}$.) The corresponding subwords are:
\begin{eqnarray*}
s_{_1} &~=&~ z_{_{s(1)+1}}  \ldots z_{_{s(2)-1}},
\\ s_{_2} &~=&~ z_{_{s(2)+1}}  \ldots z_{_{s(3)-1}},
\\ s_{_3} &~=&~ z_{_{s(3)+1}}  \ldots z_{_{s(4)-1}},
\\ s_{_4} &~=&~ z_{_{s(4)+1}}  \ldots z_{_{s(5)-1}},
\\ s_{_5} &~=&~ z_{_{s(5)+1}}  \ldots z_{_{s(1)-1}}.
\end{eqnarray*}

\begin{lem}\label{las}
Let $W$ be a cyclically reduced word of length $2m$ in a free product. Let $X$ be a reduced word of length $m$ such that both $X$ and ${X}^{-1}$ appear as cyclic subwords of $W$. Then $X$ contains a letter of order $2$.
\end{lem}
\begin{proof}
The subwords $X$ and ${X}^{-1}$ of $W$ must intersect, for otherwise $W$ is a cyclic conjugate of $X{X}^{-1}$, contradicting the fact that it is cyclically reduced. Hence there is an initial segment $Y$ of $X$ or ${X}^{-1}$ that coincides with a terminal segment of ${X}^{-1}$ or $X$ respectively. Thus $Y\equiv{Y}^{-1}$ and so $Y$ has an odd length and its middle letter has order $2$.
\end{proof}

\subsection*{More notations}

 Think of $aUbU^{-1}=z_0\cdots z_{l-1}$ as a
cyclic word satisfying a partial reflectional symmetry
using the special letters as mirrors.  Thus $U$ has \textit{mirror image} $U^{-1}$. More generally the  \textit{mirror image} of $z_j\ldots z_{_k}$ is $z_{_{l+2-j}}\ldots z_{_{l+2-k}}$ ( subscripts modulo $l$ ). Unless stated otherwise, $X$ (with or without subscript) denotes an initial or terminal segment of $U$. Similarly $Y$ (with or without subscript) will to denote an initial or terminal segment of $U^{-1}$. Also $X^{-1}$ and $Y^{-1}$  are mirror images of $X$ and $Y$ respectively.  
  Using this notation, we can express the subwords in the zones as:
\begin{eqnarray*}
s_{_1} &~=&~ X_{_1}x_{_1} Y_{_1},
\\ s_{_2} &~=&~ Y_{_2}y_{_2} X_{_2},
\\ s_{_3} &~=&~ X_{_3}x_{_3} Y_{_3},
\\ s_{_5} &~=&~ Y_{_5}y_{_5} X_{_5}.
\end{eqnarray*}
where $x_i=b$ and $y_i=a$ are the corresponding special letters. In other words if $X_i$ is a terminal segment of $s_i$ ( as in the case of $s_{_2}$), then it is an initial segment of $U$. In which case  $Y_i$ is a terminal segment of $U^{-1}$. If $s(4)=l/2$, then $s_{_4}=Y_{_4}$. Otherwise, $s_{_4}$ is a subword of $U^{-1}$ which is neither an initial nor terminal segment. Note that some of the $x_i,X_i,y_i$ and $ Y_i$ are allowed to be empty as in the case of $s_{_3}$ in Figure \ref{part_b}.

\medskip
Let $s_i'$ be the mirror image of $t_i$. Then $s_i'$ is identically equal to $s_i$ (modulo $\sim$). In particular if $s_i= X_ix_{_1} Y_i$, then $s_i'= X_i' x_{_1}' Y_i'$ where $ X_i'\sim X_i$,  $x_{_1}'\sim x_{_1}$ and  $Y_i'\sim Y_i$.

\medskip
We use $M$ to denote an initial segment of $U$ or a terminal segment of $U^{-1}$. Similarly $N$ denotes a terminal segment of $U$ or an initial  segment of $U^{-1}$. $M_{_k}^+$ is the initial segment of $U$ or terminal a terminal segment of $U^{-1}$ of length  $\ell(M)+k$. If $M$ is a subword of $s_i$, then  $M_i$  denotes the image of $M$ under $Z_i$ and $M'$ is the mirror image of $M_i$. $N_{_k}^+$, $N_j$ and $N'$ are defined similarly. 

\begin{rem}\label{sub}
One thing to note is that if $M$ is a subword of $U$, say, with $\ell(M)\geq \frac{l}{4}$, then $M_i$ can not be a subword of $U$ by Lemma \ref{od2}. So that either $M'$ is a subword of $U$ or neither a subword of $U$ nor $U^{-1}$. In the case where $M$ and $M'$ are both subwords of $U$ or $U^{-1}$, then what we call $M_k^+$ will be the union of the two. It follows that $\ell(M_k^+)>\ell(M)$ for otherwise  $M'=M$ (i.e $M'$ is also an initial segment of $U$), and hence $M_i$ is a terminal segment of $U^{-1}$, and hence there is an amalgamation of cliques, contrary to hypothesis.
\end{rem}
\begin{lem}\label{peece}
Let $M$, $M'$ and $\ell(M)+k$ be as in Remark \ref{sub}. If $\ell(M)\geq \frac{l}{4}$ and both $M$ and $M'$ are subwords of $U$ or $U^{-1}$, then $M_{_k}^+$ has period $\gamma=\ell(M_{_k}^+)-\ell(M)\leq\ell(U)-\ell(M)$.
\end{lem}

The proof of  Lemma \ref{peece} follows easily from Remarks \ref{bor} and \ref{sub}.

\begin{lem}\label{le8}
Suppose that $U$ has a period $\gamma<\ell(U)$ with $X$ and $Y$ as initial and terminal segments respectively both of length  $\gamma/2$. Then no segment of  $s_i$ is of the form ${X}y_i X^{-1}$ or $ Yx_i {Y}^{-1}$.
\end{lem}
\begin{proof}
Suppose without loss of generality that $s_i$ has ${X}y_i X^{-1}$ as a segment. Then it is identically equal to a subword of $Ux U^{-1}$. Thus  $X$ is identically equal to a subword of $U$ or $U^{-1}$. Take $W$ to be any subword of $U$ of length $\gamma$. The periodicity of $U$ implies that each of $X, {X}^{-1}$ is identically equal to a cyclic subword of $W$. By Lemma \ref{las} it follows that $X$ contains a letter of order $2$ contrary to hypothesis.
\end{proof}

\begin{lem}\label{le88}
Suppose that $W$ has period $\gamma<\ell(W)$ and no element of order $2$. Then $W$ has no subword of the form $L=wrw^{-1}$ with $\ell(w)\geq \gamma/2$. 
\end{lem}
\begin{proof}
Suppose by contradiction that $W$ has a subword of the form $L=wrw^{-1}$ with $\ell(w)\geq \gamma/2$. Let $r_{_1}$ and $r_{_2}$ be the letters $\gamma/2$ places before and after $r$ in $L$ respectively, then $r_{_1}^{-1}=r_{_2}$. But by the periodicity of $W$, $r_{_1}=r_{_2}$. Hence $W$ has an element of order $2$ contradicting hypothesis.
\end{proof}
\medskip
\begin{lem}\label{pee}
Let $X$  be an initial 
 (resp. terminal) segment of $U$ 
 of length $\ell(X)\geq \frac{l}{4}$. Suppose $s_i=Y a X$  (resp. $s_i=X a Y$)
for some terminal  (resp. initial)
segment $Y$ of $U^{-1}$. If $s_i'$ is not contained entirely in $U$, then $U$ has  period $\gamma\leq 2(\ell(U)-\ell(X))$. Furthermore if $\gamma= 2(\ell(U)-\ell(X))$, then 
 $U$ has a terminal (resp. initial) segment of
 the form $xw^{-1}z_{s(i+1)}w$
(resp. $wz_{s(i)}w^{-1}x$) for some  letter $x$ and some word $w$ with $\ell(w)=\gamma/2-1$.
\end{lem}

\begin{proof}
 By symmetry it suffices to prove the first case,
where $X$ is an initial segment of $U$.
Write $U\equiv X z_{s(i+1)} V$ for some terminal segment $V$ of $U$. Consider the mirror image $s_i'$ of $t_i$. By Lemma \ref{od2}, $t_i$ is a subword of  $Vb^\psi U^{-1}$  for some $\psi$. Hence $s_i'$ is a subword of $Ub^{-\psi}V^{-1}$.
Therefore $s_i'$  has the form $w_{_1}b^{-\psi}w_{_2}^{-1}$ for some terminal segments $w_{_1}$ and $w_{_2}$ of $U$ with $\ell(w_{_2})\leq \ell(V)=\ell(U)-(\ell(X)+1)<\frac{l}{4}-1$. 
Denote by $S$  the  terminal segment of $U$ of length 
\begin{align*}
\ell(S)=&~\ell(X)-\ell(w_{_2})-1\\
=&~\ell(s_i' )-\ell(w_{_2})-\ell(Y)-2\\
=&~\ell(w_{_1})-\ell(Y)-1.
\end{align*} 
Then $w_{_1}\equiv Ya S$ and $X\equiv Sb^{-\psi}w_{_2}^{-1}$ and so  $S$ is identically equal to an initial segment of $U$. Thus by Remark \ref{bor}, we have that $U$ has period 
\begin{align*}
\gamma=&~\ell(U)-\ell(S)\\
=&~ \ell(U)-\ell(X)+\ell(w_{_2})+1\\
\leq &~  2(\ell(U)-\ell(X)).
\end{align*} 
Finally if $\gamma=2(\ell(U)-\ell(X))$, then $1+\ell(X)+\ell(w_{_2})=\ell(U)$. Thus $V\equiv w_{_2}$ and so $U\equiv X z_{s(i+1)} V \equiv Sb^{-\psi}w_{_2}^{-1}z_{s(i+1)}w_{_2}$. The result follows by talking $w=w_{_2}$ and $x=b^{-\psi}$.

\end{proof}

\begin{lem}\label{imag}
Let $Z_i$ and $Z_j$ be distinct. Suppose that the intersection of $s_i$ and $s_j$ contains a subword of the form $L=w\phi{w}^{-1}$ where $\phi$ is a special letter. Let $L_i$ and $L_j$ be the images of $L$ under $Z_i$ and $Z_j$ respectively,
and $[L_i],[L_j]$ their $\sim$-classes in $(\Omega/\sim)^*$. Then $[L_i]$ and $[L_j]$ do not intersect as cyclic subwords of the $\sim$-class of $aUbU^{-1}$ in $(\Omega/\sim)^*$.
\end{lem}
\begin{proof}
Let $\sigma$ be the intersection of $s_i$ and $s_j$.
Since $\sigma$ is non-empty, the zones $Z_i,Z_j$ are not consecutive, so $i-j\equiv\pm 2$ mod $5$.
Without loss of generality, suppose that $j\equiv i+2$ mod $5$.  Then 
 $\sigma$ is an initial segment of $s_i$ and a terminal segment of $s_j$. So $\sigma=z_{_{s(i)+1}} \ldots z_{_{s(j+1)-1}}$ where $s_j=z_{_{s(j)+1}}\dots z_{_{s(j+1)-1}}$ and $s_i = z_{_{s(i)+1}} \dots z_{_{s(j-1)-1}}$. Choose $L$ with maximal length among all subwords of $\sigma$ of that form.
Let  $\varrho$ be the union of $s_i$ and $s_j$. Then $\varrho=z_{_{s(j)+1}} \dots z_{_{s(j-1)-1}}$ and a cyclic permutation of the clique-label 
 of $v$
has the form $(\varrho z_{_{s(j-1)}} s_{_{j-1}}z_{_{s(j)}})^2$.

\medskip
Now  define $t'_j=z_{_{t(j)+1}}\dots z_{_{t(j+1)-1}}$ and $t'_i = z_{_{t(i)+1}} \dots z_{_{t(i+1)-1}}$. 
Then $t_i'$ is identical to the image $t_i$ of $s_i$, with
the possible exception of a special letter of $t_i$: the corresponding letter of $t_i'$ is also special, and these two special letters may be different powers of $a$ (or of $b$).  Similarly, $t_j'$ agrees with $t_j$ except possibly at a special letter.

Now define $\sigma_i$ to be the
terminal segment $\sigma_i=z_{_{d+1}} \dots z_{_{t(i+1)-1}}$ of $t'_i$, where $d:=t(i+1)+s(i)-s(j+1)$ mod $l$, and $\sigma_j$ to be the initial segment
$\sigma_j=z_{_{t(j)+1}} \dots z_{_{e-1}}$ of $t'_j$, where $e:=t(j) + s(j+1)-s(i)$ modulo $l$.

Then $\sigma_i,\sigma_j$ agree with the images $L_i,L_j$ of $\sigma$ under $Z_i,Z_j$ respectively, again with the possible exception that they may differ at a special letter.

\medskip
If $[L_i]$ and $[L_j]$ coincide as cyclic subwords
of $[aUbU^{-1}]$, then $\sigma_i$ and $\sigma_j$ coincide as cyclic subwords of $aUbU^{-1}$. In this case we can form the union 
$\sigma*=t'_i \cup t'_j = z_{_{t(i)+1}} \dots z_{_{t(m)-1}}$: a cyclic subword of $aUbU^{-1}$ that is disjoint from $\sigma$ and hence contains at most one special letter. But $\sigma^*$ is identically equal to
$\varrho^{-1}$, with the possible exception of a special letter of $\sigma^*$. 
It follows from 
 Lemma \ref{od3}
that $\ell(\varrho)>l/2$. Thus $\varrho \cap \sigma* \ne \emptyset$ and so without loss of generality 
some initial segment $\tau$ of $\varrho$ coincides with a terminal segment of $\sigma^*$.  Since the only special letter of $\varrho$ is $\phi$ which does not appear in
$\sigma^*$, $\tau$ does not contain a special letter.
It follows that $\tau^{-1}\equiv\tau$, whence $\tau$ contains a letter of order $2$, contrary to hypothesis.

\medskip
Suppose then that  $[L_i]$ and $[L_j]$ intersect but do not coincide. Consider the subword  $\varrho*=\sigma_i \cup \sigma_j$ of $t'_i \cup t'_j$. As above, $\varrho^*$ is a cyclic subword
of $aUbU^{-1}$ which is disjoint from $\sigma$ and hence contains at most one special letter.
Write $\varrho*=x_{_1} x_{_2} \ldots x_{_r}$ for some $r$.
 Note that $r$ is odd, since by definition
$\sigma$ begins and ends with letters in the same free factor, and hence the same holds for $\varrho$.

Assume without loss of generality that $\sigma_i,\sigma_j$ are the initial and terminal segments respectively of $\varrho^*$.  Then
$$\sigma_i=x_{_1} x_{_2} \ldots x_{_{\ell(L)}}$$ and $$\sigma_j=x_{_{r+1-\ell(L)}} x_{_{r+2-\ell(L)}} \ldots x_{_r}.$$
Since each of $\sigma_i,\sigma_j$ agrees with $\sigma^{-1}$
except possibly at a special letter of $\varrho^*$, it follows that, for $1\le\mu\le\ell(L)$,
\begin{equation}\label{e1}
x_{_\mu}=x_{_{r+\mu-\ell(L)}}
\end{equation}
unless one of $x_\mu,x_{_{r+\mu-\ell(L)}}$
is a special letter of $\varrho^*$.

 Also, since $\sigma^{-1}$ agrees with $\sigma$ except for the middle letter $\phi^{\pm 1}$, 
 it follows that for $1\le\mu\le\ell(L)$,
\begin{equation}\label{e2}
x_\mu=x_{_{\ell(L)+1-\mu}}^{-1}
\end{equation}
 unless either  $\mu= \frac{\ell(L)+1}{2}$
 or one of $x_\mu=x_{_{\ell(L)+1-\mu}}^{-1}$
 is a special letter of $\varrho$.

Similarly, for $r+1-\ell(L)\le\mu\le r$,
\begin{equation}\label{e3}
x_\mu =x_{2r-\ell(L)-\mu}^{-1}
\end{equation}
unless either $\mu=r-\frac{\ell(L)}2$
or one of $x_\mu =x_{2r-\ell(L)-\mu}^{-1}$
is a special letter  of $\varrho^*$.

Now consider the three letters $x_{\ell(L)+1-\frac{r+1}2}$,
$x_{\frac{r+1}2}$ and $x_{\frac{3(r+1)}2-\ell(L)-1}$
of $\varrho^*$.  We know at most one of these three
letters can be special.

Suppose first that neither of $x_{\ell(L)+1-\frac{r+1}2}$,
$x_{\frac{r+1}2}$ is special.  Then

\begin{align*}
x_{_{\frac{r+1}{2}}}=&~x_{_{\ell(L)+1-\frac{r+1}{2}}}^{-1}~~\mathrm{by~(\ref{e2})}\\
=&~x_{_{r+\ell(L)+1-\frac{r+1}{2}-\ell(L)}}^{-1}~~\mathrm{by~(\ref{e1})}\\
=&~x_{_{\frac{r+1}{2}}}^{-1}
\end{align*}
It follows that $x_{_{\frac{r+1}{2}}}$ has order $2$,
contrary to hypothesis.

Similarly, if neither $x_{\frac{r+1}2}$ nor $x_{\frac{3(r+1)}2-\ell(L)-1}$ is special, then
we may deduce that $x_{\frac{r+1}2}$  using (\ref{e1})
and (\ref{e3}).

Finally, suppose that $x_{\frac{r+1}2}$ is a special letter of $\varrho^*$.  Then $x_\mu=x_{r+1-\mu}^{-1}$
for each $\mu=1,2,\dots,\frac{r-1}2$.  In particular
$$x_{\frac{\ell(L)+1}2}=x_{r+1-\frac{\ell(L)+1}2}^{-1}.$$
But
$$x_{\frac{\ell(L)+1}2}=x_{r+1-\frac{\ell(L)+1}2}$$
by (\ref{e1}).  Hence $x_{\frac{\ell(L)+1}2}$ has order $2$, again contrary to hypothesis.  This completes  the proof.

\end{proof}
\begin{rem}
We remark that the first part of the proof of Lemma \ref{imag} does not assume any form for the intersection.
\end{rem}
\subsection{Proof of Theorem B}

We are now ready to  complete the proof of Theorem \ref{thmm2}. Our proof is by contradiction. Suppose that some interior vertex $v$ of $\Gamma$  fails to satisfy $C(6)$, then we 
know from Remark \ref{remn} that $v$ has exactly five incident zones $Z_{_1},\ldots , Z_{_5}$. We also assume by  Lemma \ref{od3} 
and Remark \ref{remn} 
that $\omega_i<l/2$ for $i=1, \ldots , 5$. The proof is divided into two cases.

\medskip
Case 1.  $s(4)=l/2.$

\medskip
Case 2.  $s(4)>l/2.$

\medskip
\begin{proof}[Proof of {Case $1$}]

\medskip
For Case $1$ (alternatively case \ref{isteq1}), take $N$ to be the longer of $X_{_3}$ and $Y_{_4}$ and $M$ to be the longer of $X_{_5}$ and $Y_{_2}$. In each case $M$ is either an initial segment of $U$ or a terminal segment of $U^{-1}$. Similarly $N$ is either a terminal segment of $U$ or an initial segment of $U^{-1}$. Also $\ell(N), \ell(M)\geq \frac{l}{4}$ for otherwise $\ell(s_i)\geq l/2$ for some $i\in \lbrace 1,2, 5\rbrace$, contradicting hypothesis.

\medskip
By Lemma \ref{pee}, $U$ has an initial segment of period $\rho\leq 2(\ell(U)-\ell(M))$ and a terminal segment of period $\gamma\leq 2(\ell(U)-\ell(N))$. If $\gamma < 2(\ell(U)-\ell(N))$, then $L=z_{_{l-\frac{\gamma}{2}}} z_{_{l+1-\frac{\gamma}{2}}} \ldots z_{_{l+\frac{\gamma}{2}}}$ is a proper subword of $s_{_2}\cap s_{_5}$.  It follows from Lemma \ref{imag} that at least one of the images of $L$ under $Z_{_2}$ and $Z_{_5}$ is identically equal to  a periodic subword of $U$ (with period $\gamma $). This can not happen by Lemma \ref{le88}.

\medskip
Otherwise $U$ has period $\gamma= 2(\ell(U)-\ell(N))$. It is easy to see that $$\ell(X_{_1}),\ell(Y_{_1})\geq \ell(U)-\ell(N)$$ as otherwise either $\ell(s_{_2})\geq l/2$ or $\ell(s_{_5})\geq l/2$, thereby contradicting hypothesis. Hence the result follows by applying Lemma \ref{le8} to $s_{_1}$.

\end{proof}

For Case $2$ (alternatively case \ref{isteq2}), take $N$ to be the longer of $X_{_3}$ and $Y_{_1}$ and $M$ to be the longer of $X_{_5}$ and $Y_{_2}$. In each case $M$ is either an initial segment of $U$ or a terminal segment of $U^{-1}$. Similarly $N$ is either a terminal segment of $U$ or an initial segment of $U^{-1}$. Also $\ell(N), \ell(M)\geq \frac{l}{4}$. The proof is subdivided into three sub-cases namely: 

\medskip
Case 2a.  Each of  $M'$  or $N'$ is identically equal to a subwords of $U^{-1}$ or  $U$.

\medskip
Case 2b.  Exactly one of $M'$  or $N'$ is identically equal to a subword of $U$ or $U^{-1}$.

\medskip
Case 2c.   Both $M'$ and $N'$ are not identically equal to a subword of $U^{-1}$ or  $U$.

\begin{proof}[Proof of {Case $2a$}]
\medskip

Using Lemma \ref{peece}, we conclude that $U$ has a periodic initial segment $M_{_\gamma}^+$ of period $\gamma\leq\ell(U)-\ell(M)$ and a periodic terminal segment $N_{_\rho}^+$ of period $\rho\leq\ell(U)-\ell(N)$. Moreover by Remark \ref{remn}, the two segments intersect in a segment $S$ with $\ell(S)=\ell(M_{_\gamma}^+)+\ell(N_{_\rho}^+)-\ell(U)\geq \rho+\gamma +1>\rho+\gamma$. Hence $U$ is periodic with period $\nu=gcd(\rho,\gamma)\leq \text{min}\lbrace \rho,\gamma\rbrace\le\min\{\ell(Y_1),\ell(X_1)\} $ by Corollary \ref{ctm}. It follows that $s_{_1}$ contains a subword the form $wx_{_1}{w}^{-1}$ with $\ell(w)=\frac{\nu}{2}$. The result will then follow from Lemma \ref{le8}.
\end{proof}

\begin{proof}[Proof of {Case $2b$}]

\medskip
Suppose $N'$ is identically equal to a subword of $U$ or $U^{-1}$. By Lemma \ref{pee}, $U$ is periodic with period $\nu\leq 2(\ell(U)-\ell(M))$. If $\nu<2(\ell(U)-\ell(M))$, the result follows by applying Lemma \ref{le8} to $s_{_1}$  as in Case 2a.

\medskip Hence we may assume that $\nu=2(\ell(U)-\ell(M))$.
 Then by Lemma \ref{pee} $U$ a has a terminal segment of the form $u_{_1}wu_{_2}w^{-1}$ for some letters $u_{_1}, u_{_2}$ and some word $w$ with $\ell(w)=\frac{\nu}{2}-1$. 

Also by Lemma \ref{peece}, $U$ has a periodic terminal segment $N_{_\rho}^+$ of length $\ell(N_{_\rho}^+)=\ell(N)+ \rho$ and  period $\rho\leq\ell(U)-\ell(N)$. If $\rho +2\leq\nu$, then $N_{_\rho}^+$ has a subword of the form $\hat{w}u_{_2}\hat{w}^{-1}$ where $\hat{w}$ is the terminal segment of $w$ of length $\rho/2$.
This contradicts Lemma \ref{le88}.

Finally if $\rho +2>\nu$, then $\rho+\nu\leq 2\rho+1\leq\ell(N)+\rho$. It follows from Corollary \ref{ctm} that $U$ has period $\lambda=\text{gcd}(\nu, \rho)$. 
Hence we can apply Lemma \ref{le8} to $s_{_2}$ since $$\frac\lambda2\le\lambda-1\le\rho-1\leq \ell(U)-\max\{\ell(X_3),\ell(Y_1)\}-1\le\text{min}\lbrace X_{_2}, Y_{_2}\rbrace.$$

\medskip
The proof for the case where $M'$ is identically equal to a subword of $U$ or $U^{-1}$ is similar.
\end{proof}

\begin{proof}[Proof of {Case $2c$}]

\medskip
Then $U$ is periodic with periods $\rho\leq 2(\ell(U)-\ell(M))$ and $\gamma\leq 2(\ell(U)-\ell(N))$. If $\rho < 2(\ell(U)-\ell(M))$ or $ \gamma< \rho$, then the result follows by applying Lemma \ref{le8} to $s_{_2}$. Similarly if $\gamma < 2(\ell(U)-\ell(N))$ or $ \rho< \gamma$, then the result follows by applying Lemma \ref{le8} to $s_{_1}$.

\medskip
Hence  suppose $\rho=\gamma =2(\ell(U)-\ell(M))=2(\ell(U)-\ell(N))$.
Then we have $\ell(X_5)\le\ell(M)=\ell(U)-\gamma/2$,
whence $\ell(X_1)=\ell(U)-1-\ell(X_5)\ge\gamma/2-1$
and so $\ell(Y_1)=\ell(s_1)-1-\ell(X_1)<\ell(U)-1-\ell(X_1)\le\ell(U)
-\gamma/2=\ell(N)$.  It follows that $N=X_3$.  By a similar argument, $M=X_5$.

 It follows from Lemma \ref{pee} that $U$ has 
 an initial segment of the form  $w_{_1}z_{s(3)}w_{_1}^{-1}u_{_1}$ and a terminal segment of the form  $u_{_2}w_2z_{s(1)}^{-1}w_2^{-1}$ for some
 letters $u_{_1}=z_\rho,u_{_2}=z_{\frac{l}2-\rho}$, and words $w_{_1},w_{_2}$ satisfying $\ell(w_{_1})=\ell(w_{_2})=\frac{\rho}{2}-1$.
 
 We also have from Lemma \ref{pee} that $t(3)+\omega_3\equiv s(3) \equiv \frac\rho2$ mod $l$.
The $\frac{l-\rho}2$-th letter of $s_3$ is $x_3=z_{\frac{l}2}=b$.  The corresponding letter of $t_3$ is therefore $z_{s(3)-\frac{l-\rho}2}=z_{\frac{l}2+\rho}$,
which is not a special letter of $t_3$.  Hence $z_{\frac{l}2+\rho}=z_{\frac{l}2}^{-1}=b^{-1}$.
Therefore $u_2=z_{\frac{l}2-\rho}=z_{\frac{l}2+\rho}^{-1}=b$.
A similar argument shows that $u_1=z_\rho=z_{l-\rho}^{-1}=y_5=a$.

  Moreover $$w_{_1}z_{s(3)}w_{_1}^{-1}a=z_1\cdots z_\rho$$ is a cyclic conjugate of $$bw_{_2}z_{s(1)}^{-1}w_{_2}^{-1}=z_{\frac{l}2-\rho}\cdots z_{\frac{l}2-1},$$ by the periodicity of $U$.

Now apply Lemma \ref{nz2} to this pair of cyclically
conjugate words to obtain
$$bw_{_2}z_{s(1)}^{-1}w_{_2}^{-1}=\prod_{t=1}^s b^{\beta(t)}Vx^{\alpha(t)}V^{-1}$$
for some letter $x$, some word $V$ and some integer $s$
and some $\alpha(t),\beta(t)\in\{\pm 1\}$.  The integer $s$ in the statement
of Lemma \ref{nz2} is defined to be $k/m$, where $k=\rho/2$
and $m=\gcd(j,k)$.  Here $j$ in turn is defined as 
the number of places by which one has to cyclically permute
$bw_{_2}z_{s(1)}^{-1}w_{_2}^{-1}$ to obtain 
$w_{_1}z_{s(3)}w_{_1}^{-1}a$ -- in other words
$$z_1\cdots z_\rho \equiv z_{\frac{l}2-\rho+j}\cdots z_{\frac{l}2-1}z_{\frac{l}2-\rho}\cdots z_{\frac{l}2-\rho+j-1}.$$

\medskip
Suppose first that $s$ is even. Then Lemma \ref{nz2} gives
that $b=z_{s(1)}$, $a=z_{s(3)}^{-1}$, and $x=a$ in the above
expression.  This $U$ has a terminal segment
$$bw_{_2}z_{s(1)}^{-1}w_{_2}^{-1}=\prod_{t=1}^s b^{\beta(t)}Va^{\alpha(t)}V^{-1}$$
Similarly $U$ has an initial segment of the form
$$w_{_1}z_{s(3)}w_{_1}^{-1}a=\prod_{t=1}^s V^{-1}b^{\delta(t)}Va^{\gamma(t)}$$
Putting these together, using the periodicity of $U$, we obtain
$$U=V^{-1}\prod_{t=1}^p b^{\zeta(t)}Va^{\eta(t)}V^{-1}$$
for some integer $p$ and some $\zeta(t),\eta(t)\in\{\pm 1\}$.

Thus $UbU^{-1}\in\<a,V^{-1}bV\>$ and we have a refinement of
our generalised triangle group description of $G$, contrary to our underlying hypotheses.

\medskip
Next suppose that $s$ is odd in Lemma \ref{nz2}.  Then
$\{b,z_{s(1)}^{-1}\}=\{a,z_{s(3)}\}$, and we
have an expression
$$bw_{_2}z_{s(1)}^{-1}w_{_2}^{-1}=\prod_{t=1}^s b^{\beta(t)}Vz_{s(1)}^{\alpha(t)}V^{-1}$$
and an analogous expression
$$w_{_1}z_{s(3)}w_{_1}^{-1}a=\prod_{t=1}^s V^{-1}z_{s(3)}^{\delta(t)}Va^{\gamma(t)}$$
Again we can fit these together using the periodicity of $U$
to get an expression for $U$.  

If $b=z_{s(3)}$ and $a=z_{s(1)}^{-1}$ then again this has the form
$$U=V^{-1}\prod_{t=1}^p b^{\zeta(t)}Va^{\eta(t)}V^{-1}$$
which leads to a refinement, contrary to hypothesis.

If on the other hand $b=a$ and $z_{s(3)}=z_{s(1)}^{-1}$
then we obtain an expression of the form
$$U=Vz_{s(3)}^{\eta(0)}V^{-1}\prod_{t=1}^p a^{\zeta(t)}Vz_{s(3)}^{\eta(t)}V^{-1}\in\<a,Vz_{s(3)}V^{-1}\>$$

As before, this yields a refinement, contrary to hypothesis.

 This completes the proof.

\end{proof}

\section{Applications}\label{sec8}
Here we give the  proofs  for Theorems  \ref{corro 1}, \ref{corro 2} and \ref{corro 3}. As before we suppose that the triangle group description for $G$ is maximal. By Theorems \ref{thmm1} and \ref{thmm2}, a minimal clique-picture over $G$ satisfies the $C(6)$ property.

\begin{proof}[\textbf{Proof of Theorem \ref{corro 1}}]

\medskip
Suppose that there is a nontrivial word $w$ in  $H$, $G_{_1}$ or $G_{_2}$ that is trivial in $G$. Then we obtain a minimal
picture $P$ over $G$ on $D^2$ with boundary label $w$.  We prove the theorem by induction on the number
of cliques in $P$; the case of $0$ cliques corresponds to the empty picture $P$, for which there is nothing to prove.

Suppose first that some clique $v$ in $P$ is not simply connected, and let $C$ be one of the boundary components
of the surface carrying the clique $v$.  By an innermost argument, we may assume that $C$ bounds a disc $D\subset D^2$ such that every clique in $P\cap D$ is simply-connected.

Now the label on $C$ is a word in $H$ which is the identity in $G$, and $P\cap D$ has at least one fewer clique than $P$,
so by inductive hypothesis the label is trivial in $H$.

We then amend $P$ by replacing $P\cap D$ by a picture over $H$, all of the arcs, vertices and regions of which will
belong to the same clique as $v$ in the amended picture $P'$.  Since $C\cap P$ was not empty (for otherwise $D$ is contained in $v$), the new picture $P'$ also has fewer cliques than $P$, and the result follows from the inductive hypothesis.

Hence we are reduced to the situation where every clique in $P$ is simply connected, and hence we may form from $P$ a 
clique-picture $\Gamma$  over $G$ on $D^2$ with boundary label $w$. Without loss of generality, we may assume that
$\Gamma$ is minimal.
It follows  that $\Gamma$  satisfies $C(6)$.

\medskip
If $\Gamma$  is empty, then $w$ is already trivial in  $H$, $G_{_1}$ or $G_{_2}$ and so we get a contradiction. On the other hand suppose that $\Gamma$  is non-empty. If no arcs of $\Gamma$  meet $\partial D^2$, then  $\Gamma$  is a spherical picture (i.e a picture on $S^2$) and the $C(6)$ property implies $\chi(S^2)\neq 2$. This  contradiction implies $G_{_1}\rightarrow G$  and $G_{_2}\rightarrow G$ are both injective. 

\medskip
Suppose then that some arcs of $\Gamma$ meet $\partial D^2$. Then  $w\not\in G_{_1}, G_{_2}$. Moreover if $w$ is a word in $\lbrace a,UbU^{-1}\rbrace$, then the $C(6)$ condition combined with [\cite{LS}, Chapter $V$ Corollary 3.3] implies that some boundary  clique $v_0$ has at most degree three.

\medskip
 Under Hypothesis A, by  Corollary \ref{lz}, 
$v_0$ is connected to $\partial D^2$ by a
zone $Z_i$ with $\omega_i>l$. Either a refinement is possible by Lemma
 \ref{refine} 
or we can amalgamate $v_0$ with $\partial D^2$ to form a new clique-picture with fewer cliques  whose boundary label also belongs to $H$. The former possibility contradicts the maximality of the triangle group description for $G$ while the latter contradicts the minimality of $\Gamma$.

\medskip
Under Hypothesis B, it follows from Remark \ref{remn} that $v_0$ is connected to $\partial D^2$ by a zone $Z_i$ with $\omega_i>l/2$. 
By Lemma \ref{od3},
 $U$ has a letter of order $2$ or a refinement is possible, contradicting hypothesis. Otherwise as before we can amalgamate $v_0$ with $\partial D^2$ to form a new clique-picture with fewer cliques. Either case is a contradiction.
\end{proof}

\begin{proof}[\textbf{Proof of Theorem \ref{corro 2}}]

\medskip
Any clique-picture $\Gamma$ over $G$ satisfies $C(6)$, and hence a quadratic isoperimetric inequality (\cite{LS}, the Area Theorem of Chapter $V$). In other words, there is a quadratic function $f$ such that any word of length $m$ representing the identity element of $G$ is the boundary label of a clique-picture with at most $f(m)$ cliques. Also there is a bound (as a function of $m$) on the length of any clique-label
of $\Gamma$. By Corollary \ref{lz} and 
Lemma \ref{od3}, 
a clique with label of length $ml$ has degree at least $m$. Moreover,
there is a linear isoperimetric inequality of the form \[\ell(\Gamma)\geq \sum_{v}[\text{deg}(v)- 6]\]
where the sum is over all cliques $v$. Hence no clique can have degree greater than $\ell(\Gamma) +6f(\ell(\Gamma))$. Since no zone has length greater than $l$, this gives an upper bound of $l[\ell(\Gamma) +6f(\ell(\Gamma))]$ on the length of any clique-label.
Since both the number of cliques and the length of any clique-label are bounded, there are only a finite number of connected graphs that could arise as clique-pictures for words of length less than or equal to that of a given word $w$. Moreover, any such graph can be labelled as a clique-picture only in a finite number of ways. For any such potential labelling, we may check whether or not the clique-labels are equal to the identity in $H$, and whether or not the region-labels are equal to the identity in $G_{_1}$ or $G_{_2}$, using the solution to the word problem in $H$, $G_{_1}$ and $G_{_2}$ respectively. Hence we may obtain an effective list of all words of length less than or equal to $\ell(w)$ that appear as boundary labels of connected clique-pictures over $G$. In particular, we may check, for all cyclic subwords $w_{_1}$ of $w$, whether or not $w_{_1}g$ belongs to this list for some letter $g\in G_{_1}\cup G_{_2}$. (Note that this check also uses the solution to the word problem in $G_{_1}$ and $G_{_2}$, and that the letter $g$, if it exists, is unique by the Freiheitssatz). If so, then $w$ is a cyclic conjugate of $w_{_1}w_{_2}$ for some $w_{_2}$ , so $w = 1$ in $G$ if and only if $g=w_{_2}$ in $G$, which we may assume inductively is decidable. Hence the word problem is soluble for $G$.
\end{proof}

\begin{proof}[\textbf{Proof of Theorem \ref{corro 3}}]

\medskip
Suppose by contradiction that $W_{_1}=1=W_{_2}$ in $G$. We can assume by Theorem \ref{corro 1} that $\ell(W_{_1})\neq 1 \neq \ell(W_{_2}).$ We obtain a minimal clique-picture $\Gamma$ over $G$ with boundary label  $W_{_1}$ or $W_{_2}$. Suppose without loss of generality that $\Gamma$ has boundary label $W_{_1}$, form a new clique picture $\tilde{\Gamma}$ with boundary label $W_{_2}^{-1}$ by adding  a vertex labelled $R^{-n}$. $\tilde{\Gamma}$ has only one boundary vertex and is reduced since $\Gamma$ is minimal. It follows from [\cite{LS},  Corollary 3.4 of Chapter $V$] that $\tilde{\Gamma}$ has a single vertex or clique. Hence up to conjugacy $W_{_2}$ (and hence $W_{_1}$) is a word in $H$ which is trivial and has length strictly less than the length of $R^n$. This contradicts the Spelling Theorem, hence the result.
\end{proof}

\begin{proof}[\textbf{Proof of Theorem \ref{corro 4}}]
We construct a pushout square of aspherical
CW-complexes and embeddings
which realises Figure \ref{push-out} on fundamental groups.
The result follows from this construction.

To begin the construction, choose disjoint Eilenberg-MacLane complexes
$X_A,X_B$ of types $K(A,1),K(B,1)$ for the cyclic
groups $A,B$ respectively, and connect their base-points
by a $1$-cell $e_0$ to form a $K(A*B,1)$-complex $X_0:=X_A\cup e_0\cup X_B$.

In a similar way, we choose disjoint Eilenberg-MacLane complexes $X_{G1}=K(G_{_1},1)$ and $X_{G2}=K(G_{_2},1)$ and connect
their base-points by a $1$-cell $e_1$ to form a
$K(G_{_1}*G_{_2},1)$-complex $X_G:=X_{G1}\cup e_1\cup X_{G2}$.

The embedding $A*B\to G_{_1}*G_{_2}$ can be realised by a
continuous map $f:X_0\to X_G$.  Since each of $A,B$ is contained in a conjugate of $G_{_1}$ or of $G_{_2}$, we may assume without loss of generality that each of $X_A,X_B$ is mapped
by $f$ into $X_{G1}\sqcup X_{G2}$.  Replacing $X_{G1}$ and/or $X_{G2}$ by appropriate mapping cylinders, we may assume that 
$f$ maps $X_A\cup X_B$ injectively.

The $1$-cell $e_0$ is mapped by $f$ to a path in the homotopy
class of the word $U\in G_{_1}*G_{_2}$.  Let $X_1$ be the mapping cylinder of $f|_{e_0}:e_0\to X_G$.  Then $X_1$ is a $K(G_{_1}*G_{_2},1)$-complex, and $X_1=X_{G1}\cup X_{G2}\cup e_0\cup e_1\cup e_2$, where $e_2$ is a $2$-cell.  Indeed,
$X_1$ collapses to $X_G=X_1\setminus\{e_0,e_2\}$ across $e_2$.  Another important property of $X_1$ is that it contains $X_0$ as a subcomplex.

Now $H$ is a one-relator quotient of $A*B$, so we may
form a $K(H,1)$-complex $X_2$ from $X_0$ by adding a single $2$-cell $e_3$, together with cells in dimensions $3$ and above.

Finally, we form a complex $X$ from $X_1$ and $X_2$ by 
identifying their isomorphic subcomplexes $X_0$.
By the van-Kampen Theorem, the spaces $X_0=X_1\cap X_2$,
$X_1,X_2$ and $X=X_1\cup X_2$ realise the pushout diagram
in Figure \ref{push-out} on fundamental groups.  It remains
therefore only to show that $X$ is a $K(G,1)$ space, namely that $X$ is aspherical.

Now by Theorem \ref{corro 1}, each of the maps $G_{_1}\to G$,
$G_{_2}\to G$ and $H\to G$ is injective.  It follows that the
maps $A\to G$ and $B\to G$ are also injective.  By the Kurosh Subgroup Theorem for free products, it follows that the
kernels of the maps $A*B\to G$ and $G_{_1}*G_{_2}\to G$ are free,
and hence have homological dimension  $1$.
By \cite[Theorem 4.2]{jh1}, it suffices to prove that $\pi_2(X)=0$.

\medskip
Suppose then that $h:S^2\to X$ is a continuous map.  We show that $h$ is nullhomotopic in $X$ in a series of stages.
Up to homotopy, we may assume that $h$ maps $S^2$ into the $2$-skeleton of $X$, which consists of the $2$-skeleta of $X_{G1}$ and $X_{G2}$ together with the $1$-cells $e_0,e_1$ and the $2$-cells $e_2,e_3$.  We may also assume that $h$ is
transverse to the $2$-cells of $X$, and that the restriction
of $h$ to $h^{-1}(X^{(1)})$ is transverse to the $1$-cells of $X$.  Thus the preimages of the $2$-cells and $1$-cells of
$X$ in $S^2$ under $h$ form a picture on $S^2$.

Consider the sub-picture $\mathcal{P}$ formed by the preimages of the $2$-cell $e_3$ and the $1$-cell $e_0$.  This is a picture over
the generalised triangle group $H$.
Suppose that some component $\mathcal{P}'$ of $\mathcal{P}$ is on a non-simply-connected surface $\Sigma\subset S^2$, and let $\gamma$ be a boundary component of $\Sigma$ that separates $\Sigma$ from
a disc $D$ such that  $\partial D=\gamma$ and $h(D)$ involves
$2$-cells from $X_G$.  Then $h|_\gamma$ is null-homotopic in $X$, and hence in $X_2$ by Theorem \ref{corro 1}.  Let
$D'$ be a disc bounded by $\gamma$ and extend $h|\gamma$
to a map $D'\to X_2$, which we also call $h$ by abuse of
notation.  Then in $\pi_2(X)$ we can express the homotopy
class of $h$ as the sum of two classes $[h(D'\cup (S^2\setminus D))]$ and $[h(D\cup -D')]$, where $-D'$
denotes $D'$ with the opposite orientation.  If we can show that
each of these is nullhomotopic, then so is $h$, and the proof
is complete.

This reduces the problem to the case where every component
of $\mathcal{P}$ is a disc-picture over $H$.  Collapsing each such component to a point gives a clique-picture over $G$.
By Theorems \ref{thmm1} and \ref{thmm2} any reduced clique-picture satisfies C(6), so either the clique-picture is empty,
or an amalgamation of cliques is possible.

Amalgamation of cliques amounts to
amending $h$ by adding the class of a map $S^2\to X_2$.
Since $X_2$ is aspherical, this map is null-homotopic in $X_2$ and hence also in $X$.  Thus amalgamation of cliques does not change the homotopy-class of $h$ in $\pi_2(X)$.
After a finite number of clique-amalgamations, we are reduced to the case of an empty clique-picture.

In this case, $h(S^2)\subset X_1$.
Since $X_1$ is aspherical, $h$ is null-homotopic in $X_1$, and hence also in $X$. 

This completes the proof.

\end{proof}


\begin{thebibliography}{9}
\medskip
\medskip
\bibitem{dh1}
A. Duncan and J. Howie, Weinbaum's conjecture on unique subwords of nonperiodic words, {\em Proc. Am. Math. Soc.} {\bf 115} (1992) 947--954.

\medskip
\medskip
\bibitem{dh2}
A. Duncan and J. Howie, One-relator products with high powered relators. in {\em Geometric Group Theory, Sussex, Volume 1 (Sussex, 1991)}, 48--74, G. A. Niblo and M. A. Roller (Eds.), LMS Lecture Note Series, 181, Cambridge University Press, 1993.


\medskip
\medskip
\bibitem{eh}
M. Edjvet and J. Howie, On the abstract groups $(3,n,p;2)$, {\em J. London Math. Soc.} (2) {\bf 53} (1996) 271--288.

\medskip
\medskip
\bibitem{FW}
N. J. Fine and H. S. Wilf, Uniqueness theorem for periodic functions, \emph{Proc. Am. Math. Soc.} {\bf 16} (1965) 109--114.

\medskip
\medskip
\bibitem{jh1}
J. Howie, The quotient of a free product of groups by a single high-powered relator. I. Pictures. Fifth and higher powers, {\em Proc. London Math. Soc.} (3) {\bf 59} (1989) 507--540.

\medskip
\medskip
\bibitem{jh2}
J. Howie, The quotient of a free product of groups by a single high-powered relator. II. Fourth powers, {\em Proc. London Math. Soc.} (3) {\bf 61} (1990) 33--62.

\medskip
\medskip
\bibitem{jh3}
J. Howie, The quotient of a free product of groups by a single high-powered relator. III. The word problem.{\em  Proc. London Math. Soc.} (3) {\bf 62} (1991) 590--606.

\medskip
\medskip
\bibitem{hk}
J. Howie and N. Kopteva, The Tits alternative for generalised tetrahedron groups, {\em J. Group Theory} {\bf 9} (2006) 173--189.


\medskip
\bibitem{hs}
J. Howie and R. Shwartz, One-relator products induced from generalised triangle groups. {\em Comm. Algebra} {\bf 32} (2004) 2505--2526

\medskip
\medskip

\bibitem{LS}
R. C. Lyndon and P. E. Schupp,
{\em Combinatorial Group Theory} Springer-Verlag, (1977).

\medskip
\medskip

\bibitem{Mag1} W. Magnus, \"Uber diskontinuierliche Gruppen
mit einer definierenden Relation (Der Freiheitssatz), {\it J. reine angew. Math.} {\bf 163} (1930) 141--165.

\medskip
\medskip

\bibitem{Mag2} W. Magnus, Das Identit\"atsproblem f\"ur
Gruppen mit einer definierenden Relation,
{\it Math. Ann.} {\bf 106} (1932) 295--307.

\medskip
\medskip
\bibitem{Wein}
C. M. Weinbaum,
On relators and diagrams for groups with one defining relation.
{\em Illinois J. Math.} {\bf 16} (1972) 308--322. 
\end{thebibliography}
\end{document}